\documentclass[hidelinks,onefignum,onetabnum]{siamart250211}

\usepackage{lipsum}
\usepackage{amsfonts}
\usepackage{graphicx}
\usepackage{epstopdf}
\usepackage{algorithmic}
\ifpdf
  \DeclareGraphicsExtensions{.eps,.pdf,.png,.jpg}
\else
  \DeclareGraphicsExtensions{.eps}
\fi

\newsiamremark{remark}{Remark}
\newsiamremark{hypothesis}{Hypothesis}
\crefname{hypothesis}{Hypothesis}{Hypotheses}
\newsiamthm{claim}{Claim}
\newsiamremark{fact}{Fact}
\crefname{fact}{Fact}{Facts}

\headers{Space-Time Variational Schrödinger Evolution}{M.-S. Dupuy, V. Ehrlacher, C. Guillot}

\title{A space-time variational formulation for the many-body electronic Schrödinger evolution equation}

\author{Mi-Song Dupuy\thanks{Laboratoire Jacques-Louis Lions, Paris, France 
  (\email{mi-song.dupuy@sorbonne-universite.fr}, \url{https://msdupuy.github.io/}).}
\and Virginie Ehrlacher\thanks{École Nationale des Ponts et Chaussées, Champs-sur-Marne, France 
  (\email{virginie.ehrlacher@enpc.fr}, \url{https://team.inria.fr/matherials/team-members/virginie-ehrlacher-galland/}).}
\and Clément Guillot\thanks{École Nationale des Ponts et Chaussées, Champs-sur-Marne, France 
  (\email{clement.guillot@enpc.fr}, \url{https://team.inria.fr/matherials/team-members/clement-guillot/}).}
  }

\usepackage{amsopn}

\usepackage{fancyvrb}

\usepackage{amsmath}
\usepackage{amssymb}
\usepackage{stmaryrd}

\usepackage{xparse}	%Optional arguments

\usepackage{array}

\usepackage{algorithm}
\usepackage{algorithmic}
\usepackage{verbatim}

\usepackage{soul}

\usepackage{geometry}
\usepackage{lineno}

\usepackage{comment}
\usepackage{graphicx}
\usepackage{graphics}
\usepackage{subcaption}
\allowdisplaybreaks

\usepackage{url}

\newcommand{\NN}{\mathbb N}
\newcommand{\ZZ}{\mathbb Z}

\newcommand{\RR}{\mathbb R}
\newcommand{\CC}{\mathbb C}

\newcommand{\TT}{\mathbb T}

\newcommand{\abs}[1]{\left | #1 \right |}
\newcommand{\conjugate}[1]{\overline{#1}}

\newcommand{\fctshortdef}[3]{ #1 : #2 \mapsto #3}

\newcommand{\fctlonganonymousdef}[4]{
	\left \{
	\begin{array}{rcl}
		#1	&	\longrightarrow	&	#2	\\
		#3	&	\longmapsto	&	#4	\\
	\end{array}
	\right.
}
\newcommand{\fctlongdef}[5]{
	#1 :
	\fctlonganonymousdef{#2}{#3}{#4}{#5}
}
\newcommand{\fctrestrict}[2]{{#1}_{| #2}}

\newcommand{\set}[2]{\left \{ #1 \, : \, #2 \right \}}
\newcommand{\quickset}[1]{\left \{ #1 \right \}}
\newcommand{\setindic}[1]{\mathbf{1}_{#1}}

\newcommand{\diff}[1]{\, d #1 \,}

\DeclareMathOperator*{\argmin}{argmin}

\newcommand{\eexp}{\operatorname e}

\DeclareMathOperator{\setspan}{span}

\newcommand{\norm}[1]{\left \| #1 \right \|}
\newcommand{\dotprodbracket}[2]{\langle #1, #2 \rangle}
\newcommand{\dotprodparenthesis}[2]{\left( #1 \middle | #2 \right)}

\newcommand{\domain}[1]{D\! \left( #1 \right)}
\newcommand{\formdomain}[1]{Q\! \left( #1 \right)}

\DeclareMathOperator{\range}{Ran}

\newcommand{\antidualspace}[1]{ #1^\dagger }
\newcommand{\dualitybracket}[2]{\langle #1, #2 \rangle}

\newcommand{\adjoint}[1]{ {#1}^* }

\newcommand{\Liebracket}[2]{\left[ {#1}, {#2} \right]}

\newcommand{\Cregset}[1]{\mathcal C^{#1}}

\newcommand{\Cregsetsp}[2]{\Cregset {#1} \! \left( {#2} \right)}
\newcommand{\Czerosetsp}[1]{\Cregsetsp 0 {#1}}

\newcommand{\Cinfsetsp}[1]{\Cregsetsp \infty {#1}}

\newcommand{\Ccset}[1]{\mathcal C_c^{#1}}

\newcommand{\Ccsetsp}[2]{\Ccset {#1} \! \left( {#2} \right)}
\newcommand{\Cczerosetsp}[1]{\Ccsetsp 0 {#1}}
\newcommand{\Cconesetsp}[1]{\Ccsetsp 1 {#1}}
\newcommand{\Ccinfsetsp}[1]{\Ccsetsp \infty {#1}}

\NewDocumentCommand{\smallO}{o m}{o \IfValueTF{#1}{_{#1}}{} \! \left( #2 \right)}
\NewDocumentCommand{\bigO}{o m}{\mathcal O \IfValueTF{#1}{_{#1}}{} \! \left( #2 \right)}

\newcommand{\setclosure}[1]{\overline{#1}}

\newcommand{\laplacian}{\Delta}

\newcommand{\Lspace}[1]{L^{#1}}
\newcommand{\Lspacesp}[2]{\Lspace{#1} \! \left( #2 \right)}
\newcommand{\Lone}{\Lspace 1}
\newcommand{\Lonesp}[1]{\Lspacesp 1 {#1}}
\newcommand{\Ltwo}{\Lspace 2}
\newcommand{\Ltwosp}[1]{\Lspacesp 2 {#1}}
\newcommand{\Linf}{\Lspace \infty}
\newcommand{\Linfsp}[1]{\Lspacesp \infty {#1}}
\newcommand{\Llocspace}[1]{L_{\text{loc}}^{#1}}
\newcommand{\Llocspacesp}[2]{\Llocspace{#1} \! \left( #2 \right)}

\newcommand{\Lloctwosp}[1]{\Llocspacesp 2 {#1}}

\newcommand{\Hspace}[1]{H^{#1}}
\newcommand{\Hspacesp}[2]{H^{#1} \! \left( #2 \right)}

\newcommand{\Hone}{\Hspace 1}
\newcommand{\Honesp}[1]{\Hspacesp 1 {#1}}

\DeclareMathOperator{\Fourier}{\mathcal F}
\newcommand{\Fourierhat}[1]{\widehat{#1}}

\newcommand{\HHilb}{\mathcal H}

\newcommand{\XX}{\mathcal X}
\newcommand{\YY}{\mathcal Y}

\begin{document}

	\maketitle

\medskip

 \bfseries Statements and Declarations: \normalfont This
publication is part of a project that has received funding from the European Research Council (ERC) under the European Union’s Horizon 2020 Research and Innovation Programme – Grant Agreement 101077204.

\medskip

\begin{abstract}
We prove in this paper that the solution to the time-dependent Schrödinger equation can be expressed as the solution of a global space-time quadratic minimization problem that is amenable to Galerkin time-space discretization schemes, using an appropriate least-square formulation. The present analysis can be applied to the electronic many-body time-dependent Schrödinger equation with an arbitrary number of electrons and interaction potentials with Coulomb singularities. We motivate the interest of the present approach with two goals: first, the design of Galerkin space-time discretization methods; second, the definition of dynamical low-rank approximations following a variational principle different from the classical Dirac-Frenkel principle, and for which it is possible to prove the global-in-time existence of solutions.
\end{abstract}

% REQUIRED

% REQUIRED

\section{Introduction}

The aim of this paper is to introduce a new global space-time variational formulation of the linear evolution Schrödinger equation
\begin{equation} \label{equ: Schrodinger evolution}
\begin{cases}
    (i \partial_t - H - B(t)) u^*(t) = f(t), \quad t\in I,	\\
    u^*(0) = u_0,
\end{cases}
\end{equation}
where $H$ is a self-adjoint operator on a separable Hilbert space $\HHilb$ with domain $\domain H$, $I = (0,T)$ for some $T>0$ and $B: \overline{I} \ni t \mapsto B(t)$ is a strongly continuous family of bounded self-adjoint operators on $\HHilb$, and $f \in \Ltwosp{I, \HHilb}$. 
Throughout the paper, the equation will often be studied first without the time dependent part of the operator (i.e. with $B(t) = 0$) before extending the result to the case of nonzero $B$'s, in which case we will refer to
\begin{equation} \label{equ: Schrodinger evolutionnoB}
\begin{cases}
    (i \partial_t - H ) u^{*,0}(t) = f(t), \quad t\in I	\\
    u^{*,0}(0) = u_0.
\end{cases}
\end{equation}
which is only a particular case of \eqref{equ: Schrodinger evolution}. \\
The general setting above includes the case of the time-dependent Schrödinger equation defined on the whole space (in arbitrary large dimension) and with interaction potential with Coulomb singularities. This includes in particular the case of the time-dependent Schrödinger evolution equation associated to the many-body electronic Hamiltonian for molecules. The latter concerns the following case: consider a system of $N\in \mathbb{N}^*$ electrons in a molecule with $M\in \mathbb{N}^*$ nuclei in the Born-Oppenheimer approximation \cite{born_zur_1927, combes2005born}. Let us assume that the positions and electric charges of the nuclei are given respectively as $R_1,\ldots, R_M \in \mathbb{R}^3$ and $Z_1, \ldots, Z_M>0$. Then, $\HHilb = L^2(\mathbb{R}^{3N})$ (or, for fermions, $\HHilb = \bigwedge_{i=1}^N L^2(\mathbb{R}^3)$ the set of antisymmetric functions of $L^2(\mathbb{R}^{3N})$) and
\begin{equation} \label{eq:Schroop}
H = -\Delta - \sum_{k=1}^M \sum_{i=1}^N \frac{Z_k}{|x_i-R_k|} + \sum_{1\leq i < j \leq N} \frac{1}{|x_i - x_j|},    
\end{equation}
the quantity $|\cdot|$ denoting the euclidean norm in $\RR^3$.

\medskip

The outline of this paper is as follows.
In \Cref{sec: Preliminaries}, we introduce the space of solutions (as defined in \Cref{def: definition weak solution}) to \eqref{equ: Schrodinger evolutionnoB}:
\begin{equation} \label{equ: XX definition}
    \XX_H
    =	\set{u^{*,0} \in \Ltwosp{I, \HHilb}}{ \exists (u_0, f) \in \HHilb \times \Ltwosp{I, \HHilb} \, \text{such that $u^{*,0}$ solves \eqref{equ: Schrodinger evolutionnoB}}}.
\end{equation}
We prove that a solution to \eqref{equ: Schrodinger evolution} belongs to $\XX_H$, and that $u^*$ solution to \eqref{equ: Schrodinger evolution} can also be characterized as follows:
\begin{equation} \label{equ: Min problem E}
    \text{Find $u^* \in \XX_H$ such that} \;
    u^* = \argmin_{u \in \XX_H} E(u),
\end{equation}
where
\begin{equation}
    \forall u \in \XX_H, \quad
    E(u)
    =   \abs{u(0) - u_0}^2 + T \norm{(i\partial_t - H - B(t)) u - f}_{\Ltwosp{I, \HHilb}}^2.
\end{equation}
In \Cref{sec: Space-time variational formulation for the Schrödinger equation}, we state our main results:
first, in \Cref{th: XX0 XX equality}, we will prove that when $H = H_0 + A$ with $H_0$ a ``well known" operator and $A$ a ``perturbation", we actually have the equality $\XX_H = \XX_{H_0}$. \\
We then state in \Cref{cor: Variational formulation v} that this result, combined with \eqref{eq:XHchar}, leads to another space-time formulation of~\eqref{equ: Schrodinger evolution}:
\begin{equation} \label{equ: Min problem F}
    \text{Find $v^* \in \Honesp{I, \HHilb}$ such that} \;
    v^* = \argmin_{v \in \Honesp{I, \HHilb}} F(v),
\end{equation}
where
\begin{equation} \label{equ: F definition}
    \forall v \in \Honesp{I, \HHilb}, \quad
    F(v) = \abs{v(0) - u_0}^2 + T \norm{(i\partial_t - \eexp^{itH_0} (A + B(t)) \eexp^{-itH_0}) v - \eexp^{itH_0} f}_{\Ltwosp{I, \HHilb}}^2.
\end{equation}
Since the functionals $E$ and $F$ are quadratic, both \eqref{equ: Min problem E} and \eqref{equ: Min problem F} can be associated to a Lax-Milgram variational problem of the form
\[
    \forall u \in \XX_H, \quad a(u^*, u)  = l(u),
    \quad (\text{resp.} \; \forall v \in \Honesp{I, \HHilb}, \quad b(v^*, v)  = m(v))
\]
where $a: \XX_H \times \XX_H \to \CC$ (resp. $b: \Honesp{I, \HHilb} \times \Honesp{I, \HHilb} \to \CC$) is a continuous hermitian coercive sesquilinear form on $\XX_H \times \XX_H$ (resp. $\Honesp{I, \HHilb} \times \Honesp{I, \HHilb}$) and $l: \XX_H \to \CC$ (resp. $m: \Honesp{I, \HHilb} \to \CC)$ is a continuous linear form. More precisely, we introduce an appropriate least-square formulation~\cite{ern2016linear} of the latter equation,  satisfying the following wish list (under suitable assumptions on the Hamiltonian $H$): 
\begin{itemize}
    \item[(i)] the coercivity and continuity constants of the bilinear form $a$ should not decrease (respectively increase) faster than some inverse of polynomial (respectively polynomial) with the value of the final time $T$;
\item[(ii)] it should yield practical Galerkin space-time numerical schemes; 
\item[(iii)] it should be useful to define new dynamical low-rank approximations using a different variational principle than the classical Dirac-Frenkel principle~\cite{bachmayr2023low, bachmayr2016tensor, falco2019dirac, lubich2004variational, lubich2005variational, legeza2014tensor}. The advantage of the proposed formulation is that it is possible then to prove the global-in-time existence of some dynamical low-rank approximations, without any restrictions on the data of the problem.
\end{itemize}
%The variational formulation presented in this work satisfies these three requirements, as will be explained in the article. 
We would like to stress the fact that our present analysis covers the case of unbounded domains and interaction potentials with Coulomb singularities (see \Cref{subsec: Application to electronic Schrödinger}), which includes the case of the many-body electronic Schrödinger Hamiltonian introduced above. While our main motivation for the present work stems from the design of new global-in-time dynamical low-rank approximations, which is the object of another article~\cite{dupuy_low-complexity_nodate}, we present here some preliminary numerical results about global space-time Galerkin discretization schemes for the time-dependent Schrödinger equation associated to the variational formulation presented here.

Section~\ref{sec:low-rank} opens some perspectives about the usefulness of the proposed formulation for the definition of new dynamical low-rank approximations of the solution to the time-dependent Schrödinger equation, which can be proved to be well-posed globally in time without any restrictions on the data of the problem. As mentioned above, this line of research will be the object of a follow-up article~\cite{dupuy_low-complexity_nodate}.

Lastly, in Appendix~\ref{sec: Numerical implementation}, we propose Galerkin space-time discretization methods associated with the formulations proposed in Section~\ref{sec: Space-time variational formulation for the Schrödinger equation} and illustrate their numerical behavior.

\medskip

Similar least-square formulations exist for parabolic problems \cite{dautray_mathematical_2000}, the wave equation \cite{fuhrer2026well, hoonhout_stable_2023}, and the Navier-Stokes equation \cite{lemoine_fully_2021}.
Let us thus comment here about the state-of-the-art of global space-time discretization methods for the time-dependent Schrödinger equation. In~\cite{demkowicz2017spacetime}, a spacetime discontinuous Petrov-Galerkin (DPG) method for the linear time-dependent Schrödinger equation is proposed. Two variational formulations are proved to be well posed: a strong formulation (with no relaxation of the original equation) and a weak formulation (also called the \textit{ultra-weak formulation}, which transfers all derivatives onto test functions). However, the analysis is restricted to the case of an equation posed on a bounded domain and without an interaction potential. 
 In~\cite{gomez2022space} and~\cite{gomez2024space},  space-time discontinuous Galerkin methods for the  linear Schrödinger equation are proposed, where the equations are posed on a bounded domain, and with potentials that are bounded. 
 In these works, the focus is on setting the appropriate discontinuous Galerkin method for the Schrödinger evolution. 
 %In, a space-time ultra-weak discontinuous Galerkin discretization of the linear Schrödinger equation is introduced for bounded domains and piecewise smooth interaction potential. 
 In~\cite{hain_ultra-weak_2024}, the authors propose an ultra-weak global space-time formulation for the time-dependent Schrödinger operator with instationary Hamiltonian. An associated discretization scheme is proposed method as a Petrov-Galerkin global space-time discretization method. The scope of the latter work is however restricted to the case of bounded domains and bounded interaction potentials. Let us also mention here the works~\cite{karakashian1998space,karakashian1999space} where continuous and discontinuous methods for the nonlinear cubic Schrödinger equation are analyzed. We stress here the fact that our present analysis is restricted to the case of linear time-dependent Schrödinger equations. However, some ideas and tools used in this paper also appear in the nonlinear Schrödinger literature \cite{bourgain_global_1999}.

\section{Preliminaries} \label{sec: Preliminaries}

The aim of this section is twofold: (i) we introduce some notation which will be used throughout the article, together with the appropriate notion of weak solution for the time-dependent Schrödinger equations we consider in this work; (ii) we introduce a preliminary global time-space variational formulation of the time-dependent Schrödinger problem in the case when the Hamiltonian does not depend on time. However, we will see that this first variational formulation is not convenient to use in practice for numerical purposes, which motivates the main results we establish in the next section.

\subsection{Notations and Weak Solutions}

Throughout this paper, we fix some final time $T > 0$, and the interval $I = (0, T)$.
We also fix $(\HHilb, \dotprodbracket{\cdot}{\cdot})$ a separable Hilbert space with the associated norm $\abs \cdot = \sqrt{\dotprodbracket{\cdot}{\cdot}}$.
For a bounded operator $\mathfrak B$ defined on $\HHilb$, we define its operator norm as
\begin{equation}
    \norm{\mathfrak B} = \sup_{x \in \HHilb \setminus \quickset 0} \frac{\abs{{\mathfrak B} x}}{\abs x}.
\end{equation}
Unless stated otherwise, $H$ is a self-adjoint operator on $\HHilb$ with domain $\domain H \subset \HHilb$, and $B: \setclosure I \ni t \mapsto B(t)$ is a strongly continuous family of bounded self-adjoint operators on $\HHilb$.

We consider the Bochner space $\Ltwosp{I, \HHilb}$, which is a Hilbert space when equipped with the inner product
\begin{equation} \label{equ: Bochner inner product}
    \forall u, v \in \Ltwosp{I, \HHilb}, \quad
    \dotprodparenthesis u v_{\Ltwosp{I, \HHilb}}
    =	\int_I \diff t \dotprodbracket{u(t)}{v(t)}.
\end{equation}
The norm associated with the inner product \eqref{equ: Bochner inner product} is denoted by $\norm \cdot_{\Ltwosp{I, \HHilb}}$. See for example \cite{hytonen2016analysis} for an introduction to Bochner spaces. For any normed space $E$, we also denote by $\Cregsetsp{k}{I, E}$ (resp. $\Ccsetsp{k}{I, E}$) the space of $k$-th times continuously differentiable functions (resp. with compact support $K \subset I$) from $I$ to $E$.

\medskip

We use a standard notion of weak solutions to \eqref{equ: Schrodinger evolution} defined as follows.
Although the paper mostly focuses on time intervals of the form $(0, T)$, it will be useful for us to state the definition for more general time intervals:
\begin{definition} \label{def: definition weak solution}
    Let $J \subset \RR$ be a (possibly unbounded) open time interval such that $0 \in \setclosure J$.
    Let $u_0 \in \HHilb$ and $f \in \Ltwosp{J, \HHilb}$. \\
    An element $u \in \Lloctwosp{J, \HHilb}$ is said to be a weak solution to \eqref{equ: Schrodinger evolutionnoB} if the following two conditions are satisfied:
    \begin{enumerate}
        \item[(C1)] For any $\varphi \in \Cczerosetsp{J, \domain H} \cap \Cconesetsp{J, \HHilb}$,
        \begin{equation}
            \dotprodparenthesis{u}{(i\partial_t - H) \varphi}_{\Ltwosp{J, \HHilb}}
            =	\dotprodparenthesis{f}{\varphi}_{\Ltwosp{J, \HHilb}}.
        \end{equation}
        \item[(C2)] The equality $u(0) = u_0$ holds in $\HHilb$.
    \end{enumerate}
\end{definition}

\begin{remark}
    By Proposition \ref{prop: Weak evolution derivative}, (C1) implies that $u \in \Czerosetsp{\setclosure J, \HHilb}$, hence $u(0)$ is well defined as an element of $\HHilb$, and (C2) can be understood as a pointwise equality.
\end{remark}

\begin{remark}
    A consequence of Definition \ref{def: definition weak solution} is that, for $J = I = (0, T)$ and $u \in \Ltwosp{I, \HHilb}$, we say that $u$ is a weak solution to \eqref{equ: Schrodinger evolution} if:
    \item[(C1)] For any $\varphi \in \Cczerosetsp{J, \domain H} \cap \Cconesetsp{J, \HHilb}$,
    \begin{equation}
        \dotprodparenthesis{u}{(i\partial_t - H - B(t)) \varphi}_{\Ltwosp{I, \HHilb}}
        =	\dotprodparenthesis{f}{\varphi}_{\Ltwosp{I, \HHilb}}.
    \end{equation}
    \item[(C2)] The equality $u(0) = u_0$ holds in $\HHilb$.
\end{remark}

\subsection{Duhamel formula and functional space}

% In what follows, the main difficulty comes from the unbounded operator $H$.
% We begin by studying the case $B(t) = 0$, since the contribution of a nonzero $B(t)$ can be studied later.
% To this aim, for $u_0 \in \HHilb$ and $f \in \Ltwosp{I, \HHilb}$, let us introduce $u^{*,0}$ the unique solution (see Appendix \ref{sec: Weak solutions to the Schrödinger equation}) to \eqref{equ: Schrodinger evolutionnoB}.
% Our aim in this work is to prove that $u^*$ and $u^{*,0}$ solutions to \eqref{equ: Schrodinger evolution} and \eqref{equ: Schrodinger evolutionnoB} can be expressed as minimizers of a quadratic coercive functional.

% We introduce a first natural quadratic functional, but which is not immediately amenable to practical global space-time discretizations in the case of the many-body electronic Schrödinger problem introduced above. To this aim, we make use of the well-known Duhamel formula.

% \medskip

Let us equip the space $\XX_H$ defined in \eqref{equ: XX definition} with the inner product
\begin{equation} \label{equ: XX inner product definition}
    \forall u,v\in \XX_H,\; \dotprodparenthesis u v_{\XX_H}
    =	\dotprodbracket{u(0)}{v(0)} + T \dotprodparenthesis{(i\partial_t - H) u}{(i\partial_t - H) v}_{\Ltwosp{I, \HHilb}}.
\end{equation}
The associated norm is then denoted by $\norm u_{\XX_H} = \sqrt{\dotprodparenthesis{u}{u}_{\XX_H}}$.
The factor $T$ in \eqref{equ: XX inner product definition} is introduced for homogeneity, ensuring that the estimates for the $\Czerosetsp{I,\HHilb}$ norm have constants independent of $T$ (see Proposition~\ref{prop: XX properties}). It is an arbitrary choice, and removing it will result in an equivalent norm (with equivalence constants that depend on $T$), and will also influence the constants in \Cref{th: XX0 XX equality} and \Cref{cor: Variational formulation v}.

The essential properties of the space $\XX_H$ are summarized in the following result, the proof of which only requires to carefully work with Definition \ref{def: definition weak solution} and can be found in Appendix \ref{sec: Weak solutions to the Schrödinger equation}:
\begin{proposition} \label{prop: XX properties}
    The space $\XX_H$ equipped with the inner product $(\cdot|\cdot)_{\XX_H}$ defined in \eqref{equ: XX inner product definition} is a Hilbert space, and the application
    \begin{equation}
        \fctlonganonymousdef{\Ltwosp{I, \HHilb}}{\Ltwosp{I, \HHilb}}{u}{\left( I\ni t \mapsto \eexp^{itH}u(t) \right)}
    \end{equation}
    defines an isomorphism between $\XX_H$ and $\Honesp{I, \HHilb}$.
    In particular, $\XX_H$ is continuously embedded into $\Czerosetsp{I, \HHilb}$, and we have the estimate
    \begin{equation} \label{equ: XX continity estimate}
            \forall u \in \XX_H, \quad
        \norm u_{\Czerosetsp{I, \HHilb}} \leq \sqrt 2 \norm u_{\XX_H}.
    \end{equation}

    Furthermore, for any $f \in \Ltwosp{I, \HHilb}$, $u_0 \in \HHilb$ and $B: \overline{I} \ni t \mapsto B(t)$ a strongly continuous family of bounded self-adjoint operators acting on $\HHilb$, there exists a unique solution $u^* \in \Ltwosp{I, \HHilb}$ to \eqref{equ: Schrodinger evolution} in the sense of \Cref{def: definition weak solution}. Moreover, $u^*$ belongs to $\XX_H$ and satisfies
    \begin{equation}
        \label{equ: Duhamel formula}
        \forall t \in I, \quad
        u^*(t)
        =	\eexp^{-itH}u_0 -i \int_0^t \diff s \eexp^{-i(t-s)H} (f(s) - B(s)u^*(s)).
    \end{equation}
\end{proposition}

In other words, the first part of Proposition~\ref{prop: XX properties} states that the set $\XX_H$ can be equivalently characterized as follows:
\begin{equation}\label{eq:XHchar}
\XX_H = \set{I \ni t \mapsto \eexp^{-itH}v(t)}{v\in H^1(I, \HHilb)}.
\end{equation}

\begin{remark}
    It is easy to check that, whenever $H$ is bounded, we simply have $\XX_H = \Honesp{I, \HHilb}$, and the norms $\norm \cdot_{\XX_H}$ and $\norm \cdot_{\Honesp{I, \HHilb}}$ are equivalent.
    However, this equality does not hold when $H$ is an unbounded operator: let $\varphi \in \HHilb \setminus \domain H$, and define $\fctshortdef{\psi}{I \ni t}{\varphi}$. It is clear that $\psi \in \Honesp{I, \HHilb} \setminus \XX_H$.
    Conversely, the function $\fctshortdef{\eta}{I \ni t}{\eexp^{-itH} \varphi}$ belongs to $\XX_H \setminus \Honesp{I, \HHilb}$.
\end{remark}

\begin{remark}
    We do not include the time dependent part $B(t)$ in the definition \eqref{equ: XX definition} because, as we will prove in Theorem \ref{th: XX0 XX equality}, this would yield the same space with a different norm
    \begin{equation} %\label{equ: def N}
        \forall u \in \XX_H, \quad
        N(u) = \left( \abs{u(0)}^2 + T \norm{(i\partial_t - H - B(t))u}_{\Ltwosp{I, \HHilb}}^2 \right)^{\frac 1 2},
    \end{equation}
    which turns out to be equivalent to $\norm \cdot_{\XX_H}$.
\end{remark}

Thanks to \Cref{prop: XX properties}, one can reformulate the solution $u^{*,0}$ of the evolution equation \eqref{equ: Schrodinger evolutionnoB} (in the sense of \Cref{def: definition weak solution}) as follows: since
\begin{equation}
    \forall u \in \XX_H, \quad
    \norm{u - u^{*,0}}_{\XX_H}^2
    =   \abs{u(0) - u_0}^2 + T \norm{(i\partial_t - H) u - f}_{\Ltwosp{I, \HHilb}}^2,
\end{equation}
the solution $u^{*,0}$ of \eqref{equ: Schrodinger evolutionnoB} is equivalently the unique solution to the minimization problem
\begin{equation} \label{equ: Naive variational formulation}
    u^{*,0}
    =   \argmin_{u \in \XX_H} \left( \abs{u(0) - u_0}^2 + T \norm{(i\partial_t - H) u - f}_{\Ltwosp{I, \HHilb}}^2 \right).
\end{equation}
Although this problem is well-posed and seemingly simple, this formulation is unsatisfactory, since the only explicit characterization given for $\XX_H$ at this point involves the evolution group $\eexp^{-itH}$ (as can be seen from equation~(\ref{eq:XHchar})). In particular, finding an orthonormal basis (or a Riesz basis) of $\XX_H$ is not a trivial task.
\Cref{th: XX0 XX equality} will provide a more satisfying characterization of $\XX_H$ under some assumptions.
% In Section~\ref{sec: Space-time variational formulation for the Schrödinger equation}, we give a more satisfying characterization of $\XX_H$, which will be key in the practical global space-time discretization scheme we have in mind, in the case when the operator $H$ is a perturbation of a well-known operator $H_0$.

\medskip

In our analysis, we will need the following property of the space $\XX_H$. Let us define the space
\begin{equation} \label{equ: Def W inf}
\mathcal W^\infty_H
=   \bigcap_{k \in \NN} \Cinfsetsp{\setclosure I, \domain{H^k}}.
\end{equation}
The spaces $\domain{H^k}$ are defined recursively by $\domain{H^k} := \set{\varphi \in \domain{H^{k-1}}}{H^{k-1} \varphi \in \domain H}$ and endowed with the graph norm of $H^k$.
The following density result then holds.
\begin{proposition}
The space $\mathcal W^\infty_H$ defined in \eqref{equ: Def W inf} is a dense subspace of $\XX_H$, and for any $u, v \in \XX_H$ we have
\begin{equation} \label{equ: XX IBP}
    \int_I \diff t \dotprodbracket{(i\partial_t - H) u(t)}{v(t)}
    =	-i \left( \dotprodbracket{u(T)}{v(T)} - \dotprodbracket{u(0)}{v(0)} \right) + \int_I \diff t \dotprodbracket{u}{(i\partial_t - H) v}.
\end{equation}
\end{proposition}

\begin{proof}
We first prove that $\Cinfsetsp{\bar I, \HHilb}$ is a dense subset of $\XX_H$. Subsequently, we demonstrate that $\mathcal W^\infty_H$ is also a dense subset of $\XX_H$.

Let $u^{*,0} \in \XX_H$, $f = (i\partial_t - H) u^{*,0}$ and $u_0 = u^{*,0}(0)\in \HHilb$.
Define for any $t \in \RR$
\[
    \tilde u(t)
    =	\eexp^{-itH}u_0 -i \int_0^t \diff s \eexp^{-i(t-s)H} \setindic I(s) f(s),
\]
where $\setindic I$ is the characteristic function of the interval $I$.
As in Corollary \ref{cor: Weak solutions unbounded}, it can be checked that $\tilde u \in \Lloctwosp{\RR, \HHilb}$ is the unique solution to the evolution problem
\[
\begin{cases}
    (i\partial_t - H) \tilde u = \setindic I f,	\quad t \in \RR\\
    \tilde u(0) = u_0,
\end{cases}
\]
in the sense of Definition \ref{def: definition weak solution}.
In particular, the uniqueness of weak solutions implies that $\fctrestrict{\tilde u}{I} = u^{*,0}$.
Now, for all $n\geq 0$, let $\rho_n : \RR \to \RR$ be such that $(\rho_n)_{n\geq 0}$ is a family of mollifiers (for the $t$ variable) and
\[
    \forall t \in \RR, \quad
    u_n(t)
    =	\rho_n \ast \tilde u(t)
    =	\int_{\RR} \diff s \rho_n(s) \tilde u(t - s).
\]
The usual properties of mollifiers imply
\[
\begin{cases}
    u_n \in \Cinfsetsp{\bar I, \HHilb},	\\
    u_n \xrightarrow[n \to \infty]{} \tilde u \quad \text{in $\Ltwosp{I, \HHilb}$},	\\
    f_n = (i\partial_t - H) u_n = \rho_n \ast (\setindic I f) \xrightarrow[n \to \infty]{} f \quad \text{in $\Ltwosp{I, \HHilb}$}.
\end{cases}
\]
To check the last equality, take $\varphi \in \Cczerosetsp{I, \domain H} \cap \Cconesetsp{I, \HHilb}$ and compute
\[
\begin{aligned}
    \dotprodparenthesis{u_n}{(i\partial_t - H) \varphi}
    &=  \int_\RR \diff s \rho_n(s) \dotprodparenthesis{\tilde u(\cdot - s)}{(i\partial_t - H) \varphi}  \\
    &=  \int_\RR \diff s \rho_n(s) \dotprodparenthesis{f(\cdot - s)}{\varphi} \\
    &=  \dotprodparenthesis{\rho_n \ast f}{\varphi}.
\end{aligned}
\]
We have thus proved that $\norm{\fctrestrict{u_n}{I} - u}_{\XX_H} \to 0$. This proves that $\Cinfsetsp{\bar I, \HHilb}$ is a dense subspace of $\XX_H$.

Now let $u \in \Cinfsetsp{\setclosure I, \HHilb}$.
Let $n \geq 0$ and $\pi_{(-n, n)}(H)$ be the spectral projector on $(-n, n)$ associated with the self-adjoint operator $H$. We then define $v_n(t) = \pi_{(-n, n)}(H) u(t)$ for all $t\in I$. We easily check that for all $n\geq 0$,
\[
\begin{cases}
    v_n \in \mathcal W^\infty_H,	\\
    v_n \xrightarrow[n \to \infty]{} u \quad \text{in $\Ltwosp{I, \HHilb}$},	\\
    f_n = (i\partial_t - H) v_n = \pi_{(-n, n)}(H) f \xrightarrow[n \to \infty]{} f \quad \text{in $\Ltwosp{I, \HHilb}$},
\end{cases}
\]
which implies that $\norm{v_n - u}_{\XX_H} \to 0$, and it follows that $\mathcal W^\infty_H$ is a dense subspace of $\XX_H$.

To prove \eqref{equ: XX IBP}, we simply integrate by part for any $u, v \in \mathcal W^\infty_H$ and conclude by density since both sides are continuous with respect to $\norm \cdot_{\XX_H}$.
\end{proof}

\section{Space-time variational formulation for the Schrödinger equation} \label{sec: Space-time variational formulation for the Schrödinger equation}

In this section, after stating some preliminary results and definitions in Section~\ref{subsec: Kato smoothing}, we state an abstract condition on operators $H$ and $H_0$ which guarantees the equality $\XX_H = \XX_{H_0}$. Section~\ref{subsec: Proof main results} contains the proof of this characterization. In Section~\ref{subsec: Application to electronic Schrödinger}, we consider the case of the electronic many-body hamiltonian \eqref{eq:Schroop}, and prove that the characterization of the previous sections holds for $H_0 = -\laplacian$. In Section~\ref{subsec: Time regularity of the solution}, we also focus on the electronic Schrödinger equation to obtain some additional time regularity for the solution under the assumption that the electronic potential is smooth. Lastly, in Section~\ref{subsec: Dual formulation}, we make some additional comments on an alternative formulation which relies on the use of the Duhamel formula and draw the link between the proposed formulation and previous works in the literature.

\subsection{Preliminary results from Kato smoothing theory} \label{subsec: Kato smoothing}

Our analysis requires some tools from Kato smoothing theory.
For the sake of completeness, we gather here all the definitions and results needed for the next sections. Most of them can be found in \cite{RS4} or can be directly derived from it.

\medskip

Throughout this section, $H_0$ is a self-adjoint operator on $\HHilb$ with domain $\domain{H_0}$, and $A$ is a symmetric closed operator on $\HHilb$.

\begin{lemma}[Plancherel identity] \label{lem: Operator Plancherel}
    Let $\varphi \in L^1(\mathbb{R}, \HHilb)$. Define for all $p\in \mathbb{R}$,
    \begin{equation}
        \Fourierhat{\varphi}(p) = (2 \pi)^{-\frac 1 2} \int_\RR \diff t \eexp^{-ipt} \varphi(t).
    \end{equation}
  Then
    \begin{equation}
        \int_\RR \diff t \abs{A \varphi(t)}^2
        =	\int_{\RR} \diff p \abs{A \Fourierhat \varphi(p)}^2,
    \end{equation}
    where the integrals are set equal to $\infty$ if we do not have that $\varphi(t)$ (resp. $\Fourierhat \varphi(p)$) is in $\domain A$ for almost all $t \in \RR$.
\end{lemma}

\begin{proof}
    See \cite{RS4}, Chapter XIII, Section 3, Lemma 1, p. 142.
\end{proof}

\begin{proposition} \label{prop: Kato Smoothing}
    The two following properties are equivalent
    \begin{enumerate}
        \item[(i)] It holds that
        \[
            c_1 = \sup_{\varphi \in \HHilb,\; \abs \varphi = 1} \frac{1}{2\pi} \int_{\RR} \diff t \abs{A \eexp^{-itH_0} \varphi}^2 < \infty.
        \]
        
        \item[(ii)] It holds that $\domain{H_0} \subset \domain A$, and
        \[
            c_2 = \frac{1}{\pi} \sup_{\mu \in \CC \setminus \RR, \; \varphi\in \HHilb, \, \abs \varphi = 1} \abs{A (H_0 - \mu)^{-1} \varphi}^2 \abs{\Im \mu} < \infty,
        \]
        where $\Im \mu$ denotes the imaginary part of $\mu$.
    \end{enumerate}
    Moreover, when (i) and (ii) hold, then $c_1 = c_2$.
\end{proposition}		

\begin{proof}
    See \cite{RS4}, Chapter XIII, Section 3, Theorem XIII.25, p. 146.
\end{proof}

\begin{definition} \label{def: Kato smooth}
    If $A$ satisfies the properties (i) and (ii) from Proposition \ref{prop: Kato Smoothing}, we say that $A$ is $H_0$-smooth.
    Additionally, the common value of $\sqrt{c_1}$ and $\sqrt{c_2}$ is denoted by $\norm A_{H_0}$.
\end{definition}

From the proof of Proposition \ref{prop: Kato Smoothing}, we also extract the following lemma. 
\begin{lemma} \label{lem: Local Kato Smoothing}
    Let $H_0$ be a self-adjoint operator and $A$ a closed operator.
    For any $\varepsilon > 0$ we have
    \[
        \sup_{\varphi \in \HHilb, \; \abs \varphi = 1} \frac 1 2 \int_{\RR} \diff t \eexp^{-2\varepsilon \abs t} \abs{A \eexp^{-itH_0} \varphi}^2
        \leq	\sup_{\lambda \in \RR, \; \varphi \in \HHilb, \, \abs \varphi = 1} \varepsilon \abs{A (H_0 - \lambda \pm i\varepsilon)^{-1} \varphi}^2.
    \]
\end{lemma}

\begin{proof}
    The following proof is just an extract of the proof of Proposition \ref{prop: Kato Smoothing}. Since it is instructive, we recall it here for the sake of completeness. \\ 
    We first notice that, for a bounded operator $C$ on $\HHilb$, we have
    \[
        \forall \psi_1 \in \domain{\adjoint A}, \; \forall \psi_2 \in \mathcal H, \quad
        \dotprodparenthesis{\adjoint C \adjoint A \psi_1}{\psi_2} = \dotprodparenthesis{\adjoint A \psi_1}{C \psi_2},
    \]
    from which we conclude that, the following statements are equivalent :
    \begin{itemize}
        \item[a)] $\range C \subset \domain A$ and $\norm{AC} < \infty$
        \item[b)] It holds
        \[
            \sup_{\varphi \in \domain{\adjoint A}, \abs \varphi = 1} \abs{\adjoint C \adjoint A \varphi}
            < \infty.
        \]
    \end{itemize}
    Moreover, whenever a) and b) hold, we have the following equality:
    \[
        \sup_{\varphi \in \domain{\adjoint A}, \abs \varphi = 1} \abs{\adjoint C \adjoint A \varphi}
        =   \norm{AC}.
    \]
    Therefore, taking $C = (H_0 - \lambda - i\varepsilon)^{-1}$, we obtain
    \[
    \begin{aligned}
        c_\varepsilon
        &:= \sup_{\lambda \in \RR,\; \varphi \in \HHilb, \, \, \abs \varphi = 1} \varepsilon \abs{A (H_0 - \lambda - i\varepsilon)^{-1} \varphi}^2 \\
        &=	\sup_{\substack{\lambda \in \RR \\ \varphi \in \domain{\adjoint A}, \, \abs \varphi = 1}} \varepsilon \abs{(H_0 - \lambda + i\varepsilon)^{-1} \adjoint A \varphi}^2	\\
        &=	\sup_{\substack{\lambda \in \RR \\ \varphi \in \domain{\adjoint A}, \, \abs \varphi = 1}}
        \dotprodparenthesis{\varepsilon (H_0 - \lambda - i\varepsilon)^{-1} (H_0 - \lambda + i\varepsilon)^{-1} \adjoint A \varphi}{\adjoint A \varphi}.
    \end{aligned}
    \]
   The operator $\varepsilon (H_0 - \lambda - i\varepsilon)^{-1} (H_0 - \lambda + i\varepsilon)^{-1} = (2i)^{-1} ((H_0 - \lambda - i\varepsilon)^{-1} - (H_0 - \lambda + i\varepsilon)^{-1})$ being bounded, self-adjoint and positive, we can consider its square root $K_{\varepsilon}(\lambda)$, and the previous equality is equivalent to
    \[
        \sup_{\substack{\lambda \in \RR \\ \varphi \in \domain{\adjoint A}, \, \abs \varphi = 1}}
        \abs{K_{\varepsilon}(\lambda) \adjoint A \varphi}^2	\\
        =  \sup_{\lambda \in \RR, \, \varphi\in \HHilb,\, \abs \varphi = 1} \varepsilon \abs{A (H_0 - \lambda - i\varepsilon)^{-1} \varphi}^2
        =	c_\varepsilon^2,
    \]
    hence, for any $\lambda \in \RR$ $\range{K_\varepsilon(\lambda)} \subset \domain A$ and
    \[
        \norm{A K_\varepsilon(\lambda)}
        =   c_\varepsilon
    \]
    
    We now observe that
    \[
        -i (H_0 - \lambda - i\varepsilon)^{-1}
        =	\int_0^\infty \diff t \eexp^{- \varepsilon t} \eexp^{i\lambda t} \eexp^{-itH_0}.
    \]
    Using a similar identity for $(H_0 - \lambda + i\varepsilon)^{-1}$ and applying Lemma \ref{lem: Operator Plancherel} and the resolvent identity, we obtain, for all $\varphi\in \HHilb$,
    \[
    \begin{aligned}
        \int_\RR \diff t \eexp^{-2\varepsilon \abs t} \abs{A \eexp^{-itH_0} \varphi}^2
        &=	(2 \pi)^{-1}  \int_\RR \diff \lambda \abs{A((H_0 - \lambda - i\varepsilon)^{-1} - (H_0 - \lambda + i\varepsilon)^{-1}) \varphi}^2	\\
        &=	\frac 2 \pi \int_\RR \diff \lambda \abs{A K_\varepsilon(\lambda)^2 \varphi}^2	\\
        &\leq	\frac 2 \pi c_\varepsilon^2 \int_\RR \diff \lambda \abs{K_\varepsilon(\lambda) \varphi}^2	\\
        &\leq	\frac 2 \pi \frac{c_\varepsilon^2}{2i} \int_\RR \diff \lambda \dotprodparenthesis{((H_0 - \lambda - i\varepsilon)^{-1} - (H_0 - \lambda + i\varepsilon)^{-1}) \varphi}{\varphi}.
    \end{aligned}
    \]
    Denoting by $\diff{\mu_\varphi}$ the spectral measure of $\varphi$ for $H$, we have
    \[
    \begin{aligned}
        &\frac 2 \pi \frac{1}{2i} \int_\RR \diff \lambda \dotprodparenthesis{((H_0 - \lambda - i\varepsilon)^{-1} - (H_0 - \lambda + i\varepsilon)^{-1}) \varphi}{\varphi}	\\
        &=	\frac 2 \pi \int_\RR \diff \lambda \int_\RR \diff{\mu_\varphi}(x)
        \frac{\varepsilon}{(x - \lambda)^2 + \varepsilon^2}	\\
        &=	\frac 2 \pi \int_\RR \diff{\mu_\varphi}(x) \int_\RR \diff \lambda
        \frac{\varepsilon}{(x - \lambda)^2 + \varepsilon^2}	\\
        &=	2\int_\RR \diff{\mu_\varphi}(x) = 2 \abs{\varphi}^2,
    \end{aligned}
    \]
    and the result follows.
\end{proof}

\subsection{Global space-time formulation in $H^1(I, \HHilb)$} \label{subsec: Abstract setting}

The following theorem states an abstract condition on a couple of operators $H_0$ and $A$ on $\HHilb$ which ensures that $\XX_{H_0+A} = \XX_{H_0}$. The proof is postponed to \Cref{subsec: Proof main results}.

\medskip

More precisely, we make the following set of assumptions on the operators $H_0$ and $A$.

\medskip

\noindent {\bfseries Assumptions (A):} 
\begin{itemize}
    \item [(A1)] The operator $H_0$ is a self-adjoint operator on $\HHilb$.
    \item [(A2)] The operator $A$ is a closed symmetric operator on $\HHilb$ such that $D(H_0) \subset D(A)$.
    \item [(A3)] There exists some $\varepsilon > 0$ such that
    \begin{equation}
        \sup_{\lambda \in \RR} \norm{A (H_0 - \lambda \pm i \varepsilon)^{-1}} < 1.
    \end{equation}
\end{itemize} 

\begin{lemma}
    Let $H_0$ and $A$ be operators on $\HHilb$ satisfying (A), and $\fctshortdef{B}{\setclosure I \ni t}{B(t)}$ be a strongly continuous family of bounded self-adjoint operator on $\HHilb$. Then, for any $t \in \setclosure I$, ${H(t) = H_0 + A + B(t)}$ defined on $\domain{H(t)} := \domain{H_0}$ is self-adjoint.
\end{lemma}
\begin{proof}
    By (A3), we have $a = \norm{A(H_0 + i\varepsilon)^{-1}} < 1$.
    Let $\varphi \in \domain{H_0}$, then we have
    \[
    \begin{aligned}
        \abs{A \varphi}
        &=  \abs{A(H_0 + i\varepsilon)^{-1} (H_0 + i\varepsilon) \varphi}
        \leq    a \abs{(H_0 + i\varepsilon) \varphi}    \\
        &\leq   a\abs{H_0 \varphi} + a \abs \varepsilon \abs{\varphi},
    \end{aligned}
    \]
    i.e. $A$ is $H_0$-bounded with relative bound striclty smaller than $1$.
    Since, by (A2), $A$ is symmetric, the Kato-Rellich theorem implies that $H_0 + A$ is self-adjoint on $\domain{H_0}$.
    Since $B(t)$ is a bounded self-adjoint operator, the same statement holds for $H(t)$.
\end{proof}

\begin{remark} \label{rem: A bounded resolvent assumption}
    If $H_0$ is such that assumption (A1) is satisfied, we can already identify two cases for which assumptions (A2) and (A3) are satisfied:
    \begin{itemize}
        \item[i)] the operator $A$ is a bounded self-adjoint operator: as we have $\domain{H_0} \subset \HHilb = \domain A$ and
        \[
            \sup_{\lambda \in \RR} \norm{A (H_0 - \lambda \pm i \varepsilon)^{-1}}
            \leq	\norm A \norm{(H_0 - \lambda \pm i \varepsilon)^{-1}}
            \leq	\frac{\norm A}{\abs \varepsilon}
            \xrightarrow[\abs \varepsilon \to \infty]{}	0;
        \]

        \item[ii)] the operator $A$ is closed symmetric and $H_0$-smooth (in the sense of Definition \ref{def: Kato smooth}): as these properties ensure that $\domain{H_0} \subset \domain A$ (see Proposition~\ref{prop: Kato Smoothing}) and
        \[
            \sup_{\lambda \in \RR} \norm{A (H_0 - \lambda \pm i \varepsilon)^{-1}}
            \leq	\sqrt{\frac{c_2}{\varepsilon}}
            \xrightarrow[\abs \varepsilon \to \infty]{}	0,
        \]
        where $c_2$ is the constant appearing in Proposition \ref{prop: Kato Smoothing}.
    \end{itemize}
\end{remark}

We then have the following theorem, which is one of of our key results:
\begin{theorem} \label{th: XX0 XX equality}
    Let $H_0$ and $A$ be operators on $\HHilb$ satisfying (A). Set $H = H_0 + A$. 
    \begin{itemize}
    \item[(i)]{\bfseries Time-independent operator case:} It holds that $\XX_H = \XX_{H_0}$ and there exist constants $\alpha, C > 0$ independent of $T$ such that
    \begin{equation} \label{equ: XX0 XX norm equivalence}
        \forall u \in \XX_{H_0}, \quad
        \frac{\alpha}{1 + T} \norm u_{\XX_{H_0}} \leq \norm u_{\XX_H} \leq C (1 + T) \norm u_{\XX_{H_0}}.
    \end{equation}
    
    \item[(ii)]{\bfseries Time-dependent operator case:} Let $\fctshortdef{B}{t \in \setclosure I}{B(t)}$ be a strongly continuous family of bounded self-adjoint operators on $\HHilb$.
    The map $N: \XX_{H_0} \to \mathbb{R}_+$ defined by
    \begin{equation}
        \forall u \in \XX_{H_0}, \quad
        N(u) = \left( \abs{u(0)}^2 + T \norm{(i\partial_t - H - B(t)) u}_{\Ltwosp{I, \HHilb}}^2 \right)^{\frac 1 2},
    \end{equation}
    defines a norm on $\XX_{H_0}$, and there exist constants $\alpha, C> 0$ independent of $T$ such that
    \begin{equation} \label{equ: XX0 XXt norm equivalence}
        \forall u \in \XX_{H_0}, \quad
        \frac{\alpha}{(1 + T)^2} \norm u_{\XX_{H_0}} \leq N(u) \leq C (1 + T) \norm u_{\XX_{H_0}}.
    \end{equation}
    \end{itemize}
\end{theorem}
\begin{proof}
    See \Cref{subsubsec: Proof theorem XX}.
\end{proof}

The main advantage of Theorem \ref{th: XX0 XX equality} is that the space $\XX_{H_0}$ can be characterized by the evolution group $\eexp^{-itH_0}$. As a corollary, it is possible to use this property to reformulate a global space-time formulation for the time-dependent Schrödinger operator defined on the Hilbert space $\Honesp{I, \HHilb}$.
\begin{corollary} \label{cor: Variational formulation v}
    Let $H_0$ and $A$ be operators on $\HHilb$ satisfying (A). Set $H = H_0 + A$.
    Let $\fctshortdef{B}{t \in \setclosure I}{B(t)}$ be a strongly continuous family of bounded self-adjoint operators on $\HHilb$.
    Let $u_0 \in \HHilb$ and $f \in\Ltwosp{I, \HHilb}$. Let $u^* \in \XX_{H_0}$ be the solution to \eqref{equ: Schrodinger evolution} (in the sense of Definition \ref{def: definition weak solution}), and $\fctshortdef{v^*}{t \in I}{\eexp^{itH_0} u^*(t)}$.
    
    Then, the functional $F$ defined in \eqref{equ: F definition} is well-defined and quadratic. Furthermore, there exist constants $C, \alpha > 0$ independent of $T$ such that
    \begin{equation} \label{equ: Variational v well posedness}
        \forall v \in \Honesp{I, \HHilb}, \;
        \frac{\alpha}{(1+T)^\gamma} \norm{v - v^*}_{\Honesp{I, \HHilb}}
        \leq	\sqrt{F(v)}
        \leq	C(1+T) \norm{v - v^*}_{\Honesp{I, \HHilb}},
    \end{equation}
    with $\gamma = 1$ if $B=0$, and $\gamma = 2$ otherwise.

    In particular, it follows that $v^* \in \Honesp{I, \HHilb}$ is the unique solution to the minimization problem \eqref{equ: Min problem F}.
\end{corollary}
\begin{proof}
    See \Cref{subsubsec: Proof of cor v}
\end{proof}

\begin{remark}
    In physics, the function $v^*$ corresponds to the \emph{interaction picture} or \emph{Dirac picture} as opposed to $u^*$ which is called the \emph{Schrödinger picture}.
    In our case, the motivation behind the introduction of this $v^*$ lies in the particular dynamics that it follows, which can be understood as the dynamics of the real solution $u^*$ preconditionned by the ``well-known" free dynamics associated with $H_0$, and allows us to work with $\Honesp{I, \HHilb}$ instead of the space of mixed regularity $\XX_H$.
    Similar reasons motivate the introduction of the minimization problem \eqref{equ: Min problem F} in place of \eqref{equ: Min problem E}, the price to pay being that the operator $H + B(t)$ is replaced by the ``skewed" operator $\eexp^{itH_0}(A + B(t)) \eexp^{-itH_0}$.
\end{remark}

The global space-time numerical scheme we propose is simply a Galerkin method which consists in computing an approximation $v^{*,n}\in V_n \subset \Honesp{I, \HHilb}$ of $v^*$ solution to 
$$
v^{*,n} = \mathop{{\rm argmin}}_{v_n\in V_n}F(v_n). 
$$
Using Céa's lemma and Corollary~\ref{cor: Variational formulation v}, we then easily obtain that
$$
\| v^* - v^{*,n}\|_{H^1(I, \HHilb)} \leq \frac{C (1+T)^{\gamma +1}}{\alpha} \mathop{\inf}_{v_n \in V_n} \|v^* - v_n \|_{H^1(I, \HHilb)},
$$
which in turn implies that
$$
\| u^* - u^{*,n}\|_{\XX_{H_0}} \leq \frac{C (1+T)^{\gamma +1}}{\alpha} \mathop{\inf}_{u_n \in U_n} \|u^* - u_n \|_{\XX_{H_0}},
$$
where $u^{*,n} = \eexp^{-itH_0} v^{*,n}$ and $U_n = \eexp^{-itH_0}V_n$. For this numerical scheme to be practical, it is important to choose $V_n$ so that elements of the form $\eexp^{itH_0}v_n$ could be efficiently computed for any $v_n\in V_n$. We will give some examples of practical choices of discretization spaces $V_n$ in the following sections.

Some remarks are in order here.
\begin{remark}
    One can prove directly, as a consequence of estimate \eqref{equ: XX continity estimate time dependent} in Proposition \ref{prop: Weak solutions time dependent}, that
    \begin{equation}
        \forall v \in \Honesp{I, \HHilb}, \quad
        \norm{v - v^*}_{\Czerosetsp{I, \HHilb}}
        \leq	\sqrt{2 F(v)}.
    \end{equation}
    This estimate is actually better than the ``naive'' estimate
    \[
        \norm{v - v^*}_{\Czerosetsp{I, \HHilb}} \leq \sqrt 2 \norm{v - v^*}_{\Honesp{I, \HHilb}} \leq \frac{(1+T)}{\alpha} \sqrt{2 F(v)},
    \]
    and can be used to compute a posteriori error estimates.
\end{remark}
\begin{remark}
    Given a set $\Sigma \subset \HHilb$, and $\tilde v$ the solution to the problem
    \[
        \tilde v \in \argmin_{v \in \Honesp{I, \Sigma}} F(v),
    \]
    it follows from \eqref{equ: Variational v well posedness} that $\tilde v$ satisfies the ``Céa-type'' estimate
    \[
        \norm{\tilde v - v^*}_{\Honesp{I, \HHilb}}
        \leq    \frac{C (1+T)^{\gamma+1}}{\alpha} \min_{v \in \Honesp{I, \Sigma}} \norm{v - v^*}_{\Honesp{I, \HHilb}}.
    \]
    For instance, if $\Sigma$ is a low-rank subset (or any other approximation subset), then $\tilde v$ is a quasi-optimal approximation of $v^*$ in $\Honesp{I, \Sigma}$, and the dependency of the quasi-optimality constants on $T$ is polynomial.
    This should be compared with the error estimates that typically arise for this kind of approximation.
    For instance, time-stepping procedures on low-rank manifolds only yield results that are provably quasi-optimal, with a constant that grows exponentially in time \cite[Theorem 27]{bachmayr_iterative_2025}, \cite[Theorem 5.1]{koch2007dynamical}.
\end{remark}

\medskip

    If we make the additional assumption that $A$ is $H_0$-smooth in the sense of Definition~\ref{def: Kato smooth}, we have the following result which is a slight improvement of estimate \eqref{equ: XX0 XX norm equivalence}.

    \begin{theorem} \label{th: XX XX0 equality improved}
        Assume that $H_0$ and $A$ satisfy the set of assumptions (A). Let us assume in addition that $A$ is $H_0$-smooth. Then the operator $H = H_0 + A$ on $\HHilb$ is self-adjoint with $\domain{H} = \domain{H_0}$ and we have $\XX_H = \XX_{H_0}$. Besides, there exist constants $\alpha, C > 0$ such that
        \begin{equation} \label{equ: XX0 XX norm equivalence improved}
            \forall u \in \XX_{H_0}, \quad
            \frac{\alpha}{1 + T} \norm u_{\XX_{H_0}} \leq \norm u_{\XX_H} \leq C \sqrt{1 + T} \norm u_{\XX_{H_0}}.
        \end{equation}
    \end{theorem}
    \begin{proof}
        See \Cref{subsubsec: Proof of theorem XX improved}.
    \end{proof}

\subsection{Proofs of theoritical results of Section~\ref{subsec: Abstract setting}} \label{subsec: Proof main results}

We first state a proposition that we will use in the proofs of the results of the previous section.
\begin{proposition} \label{prop: Kato Smoothing cheap}
Assume that $H_0$ and $A$ satisfy the set of assumptions (A). 
    Then, the operator $H = A + H_0$ is self-adjoint on $\HHilb$ with domain $\domain{H} = \domain{H_0}$, and there exists $C \geq 0$ such that
    \begin{equation}
        \forall T > 0, \quad \forall \varphi \in \HHilb, \quad
        \int_{-T}^T \diff t \abs{A \eexp^{-itH_0} \varphi}^2
        \leq	C (1 + T) \abs \varphi^2,
    \end{equation}
    \begin{equation}
        \forall T > 0, \quad \forall \varphi \in \HHilb, \quad
        \int_{-T}^T \diff t \abs{A \eexp^{-itH} \varphi}^2
        \leq	C (1 + T) \abs \varphi^2.
    \end{equation}
\end{proposition}

\begin{proof}
    Let us introduce         
    \begin{equation}\label{eq:ceps}
        c_{\varepsilon} := \sup_{\lambda \in \RR} \norm{A (H_0 - \lambda \pm i\varepsilon)^{-1}} < 1,
    \end{equation}
    for some $\varepsilon >0$ such that (A3) is satisfied. 
    
    First, let us prove that $H_0 + A$ is self-adjoint on $\HHilb$ with domain $\domain{H_0}$. Using \eqref{eq:ceps},  we obtain for all $\psi \in \domain{H_0}$,
    \[
        \abs{A \psi}
        =	\abs{A (H_0 - i\varepsilon)^{-1} (H_0 - i \varepsilon) \psi}
        \leq	c_{\varepsilon}  \abs{(H_0 - i \varepsilon) \psi}
        \leq	c_{\varepsilon} \abs{H_0 \psi} + c_{\varepsilon} \varepsilon \abs \psi.
    \]
    This proves that $A$ is $H_0$-bounded with relative bound lower than $1$. Since we assumed $A$ symmetric, the self-adjointness of $H_0 + A$ is a consequence of the Kato-Rellich theorem.

    \medskip
    
    We now need to prove the estimates. We only prove the second one, since the proof for the first one is similar. For any $\lambda \in \RR$,
    \[
    (H - \lambda - i \varepsilon)^{-1}
    =	\left( (1 + A (H_0 - \lambda - i \varepsilon)^{-1}) (H_0 - \lambda - i \varepsilon) \right)^{-1}
    =	(H_0 - \lambda - i \varepsilon)^{-1} (1 + A(H_0 - \lambda - i \varepsilon)^{-1})^{-1},
    \]
    hence it follows from Neumann's inversion formula
    \[
        \abs{A (H - \lambda - i \varepsilon)^{-1}}
        =	\abs{A (H_0 - \lambda - i \varepsilon)^{-1} (1 + A (H_0 - \lambda - i \varepsilon)^{-1})^{-1}}
    \leq	\frac{c_{\varepsilon}}{1 - c_{\varepsilon}}.
    \]
    
    We can therefore apply Lemma \ref{lem: Local Kato Smoothing}, which shows that there exists a constant $C_\varepsilon \geq 0$ (which may vary along the calculations) such that
    \[
        \forall \varphi \in \mathcal H, \quad
        \int_\RR \diff t \eexp^{-2\varepsilon \abs t} \abs{A \eexp^{-itH} \varphi}^2
        \leq	C_\varepsilon \abs \varphi^2.
    \]
    Reducing the integration domain to $[0, \frac 1 \varepsilon]$, we obtain
    \[
    \forall \varphi \in \mathcal H, \quad
    \int_0^{\frac 1 \varepsilon} \diff t \abs{A \eexp^{-itH} \varphi}^2
    \leq	C_\varepsilon \abs \varphi^2.
    \]
    Now, taking any integer $n \geq 1$ and $\varphi \in \mathcal H$, we have
    \[
    \begin{aligned}
        \int_0^{\frac n \varepsilon} \diff t \abs{A \eexp^{-itH} \varphi}^2
        &=	\sum_{k=0}^{n-1} \int_{\frac k \varepsilon}^{\frac{k+1}{\varepsilon}} \diff t \abs{A \eexp^{-itH} \varphi}^2	\\
        &=	\sum_{k=0}^{n-1} \int_0^{\frac 1 \varepsilon} \diff t \abs{A \eexp^{-itH} (\eexp^{-i \frac k \varepsilon H} \varphi)}^2	\\
        &\leq	C_{\varepsilon} \sum_{k=0}^{n-1} \abs{\eexp^{-i \frac k \varepsilon H} \varphi}^2
        =	C_\varepsilon n \abs \varphi^2.
    \end{aligned}
    \]
    Applying the same process for negative times, and writing any $T \geq 0$ as $T = \frac n \varepsilon + \tau$, $0 \leq \tau < \frac 1 \varepsilon$, the estimate is proved.
\end{proof}

\subsubsection{Proof of Theorem \ref{th: XX0 XX equality}} \label{subsubsec: Proof theorem XX}
\begin{proof}[Proof of Theorem \ref{th: XX0 XX equality}]

{\bfseries Case (i):} Assume that $H = H_0 + A$.
\medskip
    Let $u \in \XX_{H_0}$. By Proposition~\ref{prop: XX properties}, there exists $v \in \Honesp{I, \HHilb}$ such that $u = \eexp^{-itH_0} v$ and $\norm{(i\partial_t - H_0) u}_{\Ltwosp{I, \HHilb}} = \norm{\partial_t v}_{\Ltwosp{I, \HHilb}}$. Therefore,
    \[
    \begin{aligned}
        \norm{Au}_{\Ltwosp{I, \HHilb}}
        &=	\norm{A\eexp^{-itH_0} \left( u(0) + \int_0^t \diff s \partial_t v(s) \right)}_{\Ltwosp{I, \HHilb}}	\\
        &\leq	\norm{A \eexp^{-itH_0} v(0)}_{\Ltwosp{I, \HHilb}} + \int_0^T \diff s \norm{\setindic{s \leq t} A \eexp^{-itH_0} \partial_t v(s)}_{\Ltwosp{I, \HHilb}}.
    \end{aligned}
    \]
    Proposition \ref{prop: Kato Smoothing cheap} shows the existence of a constant $C > 0$ such that
    \begin{equation} \label{equtmp:0}
    \begin{aligned}
        \norm{Au}_{\Ltwosp{I, \HHilb}}
        &\leq	C\sqrt{1+T} \abs{v(0)} + C\sqrt{1+T} \int_0^T \diff s \abs{\partial_t v(s)}	\\
        &\leq	C\sqrt{1+T} \abs{v(0)} + C\sqrt{1+T} \sqrt T \norm{\partial_t v}_{\Ltwosp{I, \HHilb}}	\\
        &\leq	C\sqrt{1+T} \left( \abs{u(0)}^2 + T \norm{(i\partial_t - H_0) u}_{\Ltwosp{I, \HHilb}}^2 \right)^{\frac 1 2}	\\
        &\leq	C\sqrt{1+T} \norm u_{\XX_{H_0}}.
    \end{aligned}
    \end{equation}
    Thus, $(i\partial_t - H)u = (i\partial_t - H_0)u - Au \in \Ltwosp{I, \HHilb}$, that is, $u \in \XX_H$ and
    \begin{equation} \label{equtmp:1}
    \begin{aligned}
        \norm u_{\XX_H}
        &\leq	\abs{u(0)} + \sqrt T \norm{(i\partial_t - H_0 - A) u}_{\Ltwosp{I, \HHilb}}	\\
        &\leq	\abs{u(0)} + \sqrt T \norm{(i\partial_t - H_0) u}_{\Ltwosp{I, \HHilb}} + \sqrt T \norm{Au}_{\Ltwosp{I, \HHilb}}	\\
        &\leq	C (\norm u_{\XX_{H_0}} + \sqrt T \sqrt{1+T} \norm u_{\XX_{H_0}})	\\
        &\leq	C(1+T) \norm u_{\XX_{H_0}}.
    \end{aligned}
    \end{equation}
    This proves the right inequality in \eqref{equ: XX0 XX norm equivalence}.
    
    We now turn to the left inequality in \eqref{equ: XX0 XX norm equivalence}. Let $u \in \XX_H$. Applying Proposition \ref{prop: Kato Smoothing cheap}, we obtain the existence of a $C > 0$ such that
    \[
        \forall \varphi \in \HHilb, \quad
        \norm{A \eexp^{-itH} \varphi}_{\Ltwosp{I, \HHilb}}
        \leq	C \sqrt{1 + T} \abs \varphi.
    \]
    The rest of the proof can be carried out without change, simply by permuting $H_0$ and $H$.
    \ \\ \
    
{\bfseries Case (ii):} We now turn to the case $H(t) = H_0 + A + B(t)$.

\medskip

    Let $u \in \XX_{H_0}$, then
    \[
    \begin{aligned}
        N(u)^2
        &=	\abs{u(0)}^2 + T \norm{(i\partial_t - H_0 - A - B(t))u(t)}_{\Ltwosp{I, \HHilb}}^2	\\
        &\leq	\abs{u(0)}^2 + 2 T \norm{(i\partial_t - H_0 - A)u(t)}_{\Ltwosp{I, \HHilb}}^2 + 2 T M^2 \norm u_{\Ltwosp{I, \HHilb}}^2.
    \end{aligned}
    \]
    The sum of the two first terms can be estimated by the first part of Theorem \ref{th: XX0 XX equality} as follows,
    \[
        \abs{u(0)}^2 + 2 T \norm{(i\partial_t - H_0 - A)u(t)}_{\Ltwosp{I, \HHilb}}^2
        \leq    C(1+T)^2 \norm u_{\XX_{H_0}}^2
    \]
    and for the third term we write
    \[
        \norm u_{\Ltwosp{I, \HHilb}}
        \leq	\sqrt T \norm u_{\Czerosetsp{I, \HHilb}}
        \leq	\sqrt{2T} \norm u_{\XX_{H_0}}.
    \]
    Therefore,
    \[
        N(u)^2
        \leq	C(1+T)^2 \norm u_{\XX_{H_0}}^2,
    \]
    which is the right-hand estimate of \eqref{equ: XX0 XXt norm equivalence}.
    
    Conversely, let $u \in \XX_{H_0}$ and write $u_0 = u(0)$ and $f(t) = (i\partial_t - H_0 - A - B(t))u(t)$. Then it follows from the continuity estimate in Proposition \ref{prop: Weak solutions time dependent} that
    \[
        \norm u_{\Ltwosp{I, \HHilb}}^2
        \leq    T \norm u_{\Czerosetsp{I, \HHilb}}^2
        \leq 2 T (\abs{u(0)}^2 + T \norm f_{\Ltwosp{I, \HHilb}}^2)
        =   2T N(u)^2,
    \]
    therefore,
    \[
    \begin{aligned}
        \norm u_{\XX_H}^2
        &\leq	2(N(u)^2 + M^2 T \norm u_{\Ltwosp{I, \HHilb}}^2)	\\
        &\leq	C (1+T^2) N(u)^2.
    \end{aligned}
    \]
    Using the first part of the proof, we finally obtain
    \[
        \norm u_{\XX_{H_0}} \leq C(1+T) \norm u_{\XX_H} \leq C(1+T)^2 N(u),
    \]
    which concludes the proof.
\end{proof}

\subsubsection{Proof of Corollary \ref{cor: Variational formulation v}} \label{subsubsec: Proof of cor v}
\begin{proof}[Proof of Corollary \ref{cor: Variational formulation v}]
    Let $v = \eexp^{itH_0} u \in \Honesp{I, \HHilb}$, then
    \[
    \begin{aligned}
        (i\partial_t - \eexp^{itH_0} A \eexp^{-itH_0}) v - \eexp^{itH_0} f
        &=	i\partial_t \eexp^{itH_0} u - \eexp^{itH_0} A u - \eexp^{itH_0} f	\\
        &=	\eexp^{itH_0}(i\partial_t - H_0) u - \eexp^{itH_0} A u - \eexp^{itH_0} (i\partial_t - H) u^*	\\
        &=	\eexp^{itH_0} (i\partial_t - H)(u - u^*).
    \end{aligned}
    \]
    It immediately follows that
    \[
    F(v) = \norm{u - u^*}_{\XX_H}^2,
    \]
    and the result is a consequence of \eqref{equ: XX0 XX norm equivalence} and the fact that $\norm{u - u^*}_{\XX_{H_0}} = \norm{v - v^*}_{\Honesp{I, \HHilb}}$.
\end{proof}

\subsubsection{Proof of Theorem \ref{th: XX XX0 equality improved}} \label{subsubsec: Proof of theorem XX improved}
\begin{proof}[Proof of Theorem \ref{th: XX XX0 equality improved}]
    To prove Theorem \ref{th: XX XX0 equality improved}, notice that if $A$ is $H_0$-smooth, then \eqref{equtmp:0} in the proof above can be replaced by
    \[
        \norm{Au}_{\Ltwosp{I, \HHilb}}
            \leq	C \norm u_{\XX_{H_0}}.
    \]
    The result is proved by adapting \eqref{equtmp:1}.
\end{proof}

\subsection{Application to electronic many-body Schrödinger operators} \label{subsec: Application to electronic Schrödinger}
    
In this section, we show that the variational formulation proposed in the previous section can also be applied to electronic many-body Schrödinger operators of the form (\ref{eq:Schroop}).

More precisely, let $H_0 = -\laplacian$, and let $A = V$ denote the multiplication by the electronic potential
\begin{equation} \label{equ: Coulomb potential}
    \forall x_1, ..., x_N \in \RR^3, \quad
    V(x_1, ..., x_N)
    =	\sum_{k=1}^M \sum_{l=1}^N \frac{-Z_k}{\abs{x_{\ell} - X_k}} + \sum_{1 \leq k < \ell \leq N} \frac{1}{\abs{x_k - x_{\ell}}},
\end{equation}
where the $Z_k>0$ denote the charges of the nuclei.

\medskip

The fundamental result needed here concerns the one-body case ($N=1$), and is proven in \cite{Burq-Strichartz-Critical-Potential, Kato-Yajima-Examples-of-smooth-operators}:
\begin{proposition} \label{prop: Coulomb Kato Smoothing for one particle}
    For any $\varphi \in \Ltwosp{\RR^3}$, we have:
    \begin{equation}
        \norm{\frac{1}{\abs x} \eexp^{it\laplacian} \varphi}_{\Ltwosp{\RR_t \times \RR^3_x}}
        \leq	2 \sqrt{\frac 2 \pi} \norm{\varphi}_{\Ltwosp{\RR^3}}.
    \end{equation}
    In other words, the multiplication operator by $\frac{1}{\abs x}$ is $-\laplacian$-smooth.
\end{proposition}

Using this result, we can obtain a similar one for the many-body case:
\begin{proposition}\label{th: Coulomb Kato smoothing}
    For any $\varphi \in \Ltwosp{\RR^{3N}}$
    \begin{equation} \label{equ: Coulomb Kato smoothing ineg}
        \norm{V\eexp^{it\laplacian}\varphi}_{\Ltwosp{\RR_t \times \RR^{3N}_x}}
        \leq	2 \sqrt{\frac 2 \pi} \left( N \sum_{k=1}^M \abs{Z_k} + \frac{N(N-1)}{2\sqrt 2} \right) \norm{\varphi}_{\Ltwosp{\RR^{3N}}},
    \end{equation}
    where $V$ is defined in (\ref{equ: Coulomb potential}).
    In other words, the multiplication operator by $V$ is $-\laplacian$-smooth in the sense of \Cref {def: Kato smooth}.
\end{proposition}
\begin{proof}
    We first consider the case $V = \frac{1}{\abs{x_1}}$.    
    Let $\varphi \in \Ltwosp{\RR^{3N}}$.
    Using the fact that $\Ltwosp{\RR^{3N}} = \Ltwosp{\RR^{3}} \otimes \Ltwosp{\RR^{3(N-1)}}$, introducing the singular value decomposition of $\varphi$, there exist two orthogonal sequences $(v_k)_{k\geq 1} \subset \Ltwosp{\RR^3}$, $(w_k)_{k\geq 1} \subset \Ltwosp{\RR^{3(N-1)}}$ such that for almost all $x_1\in \mathbb{R}^3$ and $X=(x_2, ..., x_N)\in  \mathbb{R}^{3(N-1)}$
    \[
    \varphi(x_1, X)
    =	\sum_{k\geq 1} v_k(x_1) w_k(X),
    \]
    the sum above being possibly infinite. Therefore,
    \[
    \begin{aligned}
        \norm{\frac{\eexp^{it\laplacian} \varphi}{\abs{x_1}}}_{\Ltwosp{\RR \times \RR^{3N}}}^2
        &=	\norm{\sum_{k\geq 1} \left( \frac{\eexp^{it\laplacian_{x_1}} v_k}{\abs{x_1}} \right) \otimes \left( \eexp^{it\laplacian_X} w_k \right) }_{\Ltwosp{\RR \times \RR^{3N}}}^2	\\
        &=	\sum_{k\geq 1} \norm{\left( \frac{\eexp^{it\laplacian_{x_1}} v_k}{\abs{x_1}} \right) \otimes \left( \eexp^{it\laplacian_X} w_k \right)}_{\Ltwosp{\RR \times \RR^{3N}}}^2,
    \end{aligned}
    \]
    where we used the fact that, for each $t$, $(\eexp^{it\laplacian_X}w_k)_k$ is an orthogonal family of $\Ltwosp{\RR^{3N}}$. Thus,
    \[
    \begin{aligned}
        \norm{\frac{\eexp^{it\laplacian} \varphi}{\abs{x_1}}}_{\Ltwosp{\RR \times \RR^{3N}}}^2
        &=	\sum_{k\geq 1}\int_{\RR}\diff t \norm{\frac{\eexp^{it\laplacian_{x_1}} v_k}{\abs{x_1}}}_{\Ltwosp{\RR^3}}^2 \norm{\eexp^{it\laplacian_X} w_k}_{\Ltwosp{\RR^{3(N-1)}}}^2	\\
        &\leq	\sum_{k\geq 1} \left( 2 \sqrt{\frac 2 \pi} \right)^2 \norm{v_k}_{\Ltwosp{\RR^3}}^2 \norm{w_k}_{\Ltwosp{\RR^{3(N-1)}}}^2	\\
        &=	\left( 2 \sqrt{\frac 2 \pi} \right)^2 \norm{\varphi}_{\Ltwosp{\RR^{3N}}}^2.
    \end{aligned}
    \]

    For the general case, we consider each term of the sum \eqref{equ: Coulomb potential} separately:
    \begin{itemize}
        \item for the terms $\frac{-Z_k}{\abs{x_\ell - X_k}}$, we perform the change of variables $y_\ell = x_\ell - X_k$, and replace $x_1$ by $y_\ell$ in the computation above;
        \item for the terms $\frac{1}{\abs{x_k - x_\ell}}$, we use the change of variables  $y_k = \frac{x_k - x_{\ell}}{\sqrt 2}$ and $y_{\ell} = \frac{x_k + x_{\ell}}{\sqrt 2}$, and replace $x_1$ by $y_k$ in the computation above.
    \end{itemize}
    The inequality \eqref{equ: Coulomb Kato smoothing ineg} then follows by applying the triangle inequality and combining the inequalities obtained for each term.
    % All the other cases can be reduced to this one through a unitary change of variables :
    % For $\frac{1}{\abs{x_{\ell} - X_k}}$ we set $y_j = x_j$ for $j \neq l$ and $y_{\ell} = x_{\ell} - X_k$.
    % For $\frac{1}{x_k - x_{\ell}}$ we set $y_j = x_j$ for $j \not \in \quickset{k, l}$, $y_k = \frac{x_k - x_{\ell}}{\sqrt 2}$ and $y_{\ell} = \frac{x_k + x_{\ell}}{\sqrt 2}$ (which explains why a $\frac{1}{\sqrt 2}$ appears with these terms).
\end{proof}

% \begin{remark}
%     If the potential $V$ has the slightly more general form
%     \[
%         \forall x_1, ..., x_N \in \RR^3, \quad
%         V(x_1, ..., x_N)
%         =	\sum_{k=1}^M \sum_{\ell=1}^N \frac{-a_k}{\abs{x_{\ell} - X_k}} + \sum_{1 \leq k < \ell \leq N} \frac{b_{kl}}{\abs{x_k - x_{\ell}}}
%     \]
%     where $(a_k)_k$, $(b_{kl})_{kl}$ are real numbers (with no sign constraint), we can prove (using exactly the exact same arguments as for Proposition \ref{th: Coulomb Kato smoothing}) that for any $\varphi \in \Ltwosp{\RR^{3N}}$
%     \begin{equation}
%         \norm{V\eexp^{it\laplacian}\varphi}_{\Ltwosp{\RR_t \times \RR^{3N}_x}}
%         \leq	2 \sqrt{\frac 2 \pi} \left( N \sum_{k=1}^M \abs{a_k} + \frac{1}{\sqrt 2} \sum_{1 \leq k < \ell \leq N} \abs{b_{k\ell}} \right) \norm{\varphi}_{\Ltwosp{\RR^{3N}}}.
%     \end{equation}
% \end{remark}

For any $u_0\in \Ltwosp{\mathbb{R}^{3N}}$ and any $f\in \Ltwosp{I, \Ltwosp{\mathbb{R}^{3N}}}$, we can therefore reformulate the solution to the evolution problem
\begin{equation} \label{equ: Schrodinger electronic evolution}
\begin{cases}
    i\partial_t u^* = (-\laplacian + V) u^* + f, \\
    u(0) = u_0,
\end{cases}
\end{equation}
as in the previous section (Theorem~\ref{th: XX XX0 equality improved}):
\begin{theorem}\label{th:Schro}
    Let $u_0 \in \Ltwosp{\RR^{3N}}$ and $f \in \Ltwosp{I, \Ltwosp{\RR^{3N}}}$. Let $u^*$ be the solution to \eqref{equ: Schrodinger electronic evolution}, and $v^* := \eexp^{-it\laplacian} u^*$.

    Define for any $v \in \Honesp{I, \Ltwosp{\RR^{3N}}}$ the functional
    \begin{equation}
        F(v)
        =   \norm{v(0) - u_0}_{\Ltwosp{\RR^{3N}}}^2 + T \norm{(i\partial_t - \eexp^{-it\laplacian} V \eexp^{it\laplacian}) v - \eexp^{-it\laplacian} f}_{\Ltwosp{I, \Ltwosp{\RR^{3N}}}}^2.
    \end{equation}
    Then, there exist constants $C, \alpha > 0$ such that for any $v \in \Honesp{I, \Ltwosp{\RR^{3N}}}$,
    \begin{equation}
        \frac{\alpha}{1+T} \norm{v - v^*}_{\Honesp{I, \Ltwosp{\RR^{3N}}}}
        \leq    \sqrt{F(v)}
        \leq    C \sqrt{1+T} \norm{v - v^*}_{\Honesp{I, \Ltwosp{\RR^{3N}}}}.
    \end{equation}
\end{theorem}

\begin{remark} \label{rem: general density}
    Let us consider the one-body case $N = 1$.
    In practice, the Coulomb potential may be regularized for valence electrons calculations, taking into account the frozen core orbitals. This is the effective core potential approach \cite{hay1985ab} similar to pseudopotentials in solid state physics. Schematically, at a rough level it may be described as follows:
    let $\rho \in \Lonesp{\RR^3}$ be a smooth and rapidly decaying function, and define $V_\delta = \rho_\delta \ast \frac{1}{\abs x}$, where $\rho_\delta(x) = \delta^{-d} \rho(\frac x \delta)$ represents a smeared charge distribution with a regularization parameter $\delta$
    % In practice, for numerical purposes, it often happens that the sharp Coulomb potential is replaced by a regularized potential $V_\delta = \rho_\delta \ast \frac{1}{\abs x}$, with some $\rho \in \Lonesp{\RR^3} \cap \Linfsp{\RR^3}$ and $\rho_\delta(x) = \delta^{-d} \rho(\frac x \delta)$ representing a smeared charge distribution with a regularization parameter $\delta$.
    Here, $V_\delta \in \Linfsp{\RR^3}$, therefore Corollary \ref{cor: Variational formulation v} can still be applied.
    However, we then have \mbox{$\norm{V_\delta}_{\Linf} \to \infty$} as $\delta \to 0$, hence the constants $C, \alpha$ in Corollary \ref{cor: Variational formulation v} will degenerate. \\
    An alternative approach can be derived by observing that, for any $\varphi \in \Ltwosp{\RR^3}$, one has (taking advantage of the translation invariance of $\laplacian$)
    \[
    \begin{aligned}
        \frac{1}{\sqrt{2\pi}} \norm{V_\delta \eexp^{it\laplacian} \varphi}_{\Ltwosp{\RR_t \times \RR^3_x}}
        &=  \frac{1}{\sqrt{2\pi}} \norm{\int_{\RR^3} \diff y\rho_\delta(y) \frac{1}{\abs{x - y}} (\eexp^{it\laplacian} \varphi)(x)}_{\Ltwosp{\RR_t \times \RR^3_x}}  \\
        &\leq   \int_{\RR^3} \diff y \abs{\rho_{\delta}(y)} \frac{1}{\sqrt{2\pi}}  \norm{\frac{1}{\abs{x - y}} (\eexp^{it\laplacian} \varphi)(x)}_{\Ltwosp{\RR_t \times \RR^3_x}}  \\
        &\leq   \int_{\RR^3} \diff y \abs{\rho_{\delta}(y)} \norm{\frac{1}{\abs x}}_{\laplacian} \norm \varphi_{\Ltwo}  \\
        &\leq   \norm{\rho_\delta}_{\Lone} \norm{\frac{1}{\abs x}}_{\laplacian} \norm \varphi_{\Ltwo}
        =   \norm{\rho}_{\Lone} \norm{\frac{1}{\abs x}}_{\laplacian} \norm \varphi_{\Ltwo}.
    \end{aligned}
    \]
    We recall here that $\norm \cdot_\laplacian$ is a special case of \Cref{def: Kato smooth} with $H_0 = \laplacian$.
    It follows that $V_\delta$ is $-\laplacian$-smooth, and the value of $\norm{V_\delta}_{\laplacian}$ is bounded uniformly with respect to $\delta$. Hence, the constants $C, \alpha$ in Corollary \ref{cor: Variational formulation v} can in fact be taken independent of $\delta$. \\
    This shows that, even for bounded potentials, the machinery developed in \Cref{subsec: Kato smoothing}, \Cref{subsec: Abstract setting}, and \Cref{th: Coulomb Kato smoothing} may provide some improvement compared to the more elementary approach used for bounded operators. The same argument can be applied to the many-body case.
\end{remark}

\begin{remark}
    To the knowledge of the authors, the arguments developed to prove \Cref{th: XX0 XX equality} do not extend to the case of unbounded time dependent potentials.
    For instance, the important case of an ultrashort laser pulse modeled by a time-dependent potential of the form $B(t) = -x \cdot E(t)$ is not covered by the results of the present paper.
\end{remark}

\subsection{Time regularity of the solution} \label{subsec: Time regularity of the solution}
    
In this section, we study the regularity of the function $v^*$ defined in Corollary \ref{cor: Variational formulation v} with respect to the time variable.
We consider the case $H_0 = -\laplacian$, $A=0$ and $B(t) = V(t)$ a smooth potential, but assumptions could be replaced by more abstract assumptions regarding the commutators $\Liebracket{-\laplacian}{V(t)}$, $\Liebracket{-\laplacian}{\Liebracket{-\laplacian}{V(t)}}$, ...

\begin{theorem} \label{th: Hamiltonian smooth potential}
    Let $k \in \NN$.
    Assume $V = V(t, x) \in \Cinfsetsp{\RR \times \RR^d}$ is real-valued and
    \begin{equation} \label{equ: Potential smoothness assumption}
        \mu_k :=
        \sup_{\substack{0 \leq l \leq k \\ 0 \leq \abs \alpha \leq 2k}}
        \sup_{(t, x) \in \RR \times \RR^d} \abs{\partial_t^l \partial_x^\alpha V(t, x)} < \infty.
    \end{equation}
    Let $u_0 \in \Hspacesp{k}{\RR^d}$, and $u^* \in \XX_{- \laplacian}$ be the solution to
    \begin{equation}
        \begin{cases}
            i\partial_t u^* = (-\laplacian + V) u^*,	\\
            u^*(0) = u_0,
        \end{cases}
    \end{equation}
    in the sense of Definition \ref{def: definition weak solution} (with $J = \RR)$, and define $v^* := \eexp^{-it\laplacian} u^*$.
    Then $v^* \in \Cregsetsp{k+1}{\RR, \Ltwosp{\RR^d}}$, and there exists a constant $C_k$, that only depends on $\mu_k$ such that %$k$ and the $\sup_{(t, x) \in \RR \times \RR^d} \abs{\partial_t^j \partial_x^\alpha V(t, x)}$ such that
    \begin{equation}
        \forall t \in \RR, \;
        \norm{(i\partial_t)^{k+1} v^*}_{\Czerosetsp{\RR, \Ltwosp{\RR^d}}} \leq C_k (1+ \abs t^{\frac k 2})\norm{u_0}_{\Hspacesp{k}{\RR^d}}.
    \end{equation}
\end{theorem}

The proof of Theorem \ref{th: Hamiltonian smooth potential} is based upon the following observation: if $V$ satisfies \eqref{equ: Potential smoothness assumption}, then for any given $t \in \RR$ and $k \geq 1$, the form domains $\formdomain{(-\laplacian + V(t))^k}$ and $\formdomain{(-\laplacian)^k} = \Hspacesp{k}{\RR^d}$ coincide, and the associated norms are equivalent, that is, there exist constants $a_k, M_k > 0$ such that
\begin{equation} \label{Smooth potential form domain equivalence}
    \forall \varphi \in \Hspacesp{k}{\RR^d}, \quad
    \frac{1}{M_k} \norm \varphi_{\Hspacesp{k}{\RR^d}}^2
    \leq	\dotprodbracket{(-\laplacian + V(t))^k \varphi}{\varphi} + a_k \norm \varphi_{\Ltwosp{\RR^d}}^2
    \leq	M_k \norm \varphi_{\Hspacesp{k}{\RR^d}}^2.
\end{equation}
Moreover, the constants $a_k, M_k$ only depend on $\sup_{x \in \RR^d} \abs{\partial_x^\alpha V(t, x)}$ for $\abs \alpha \leq k$, hence can be chosen independent of $t$ according to \eqref{equ: Potential smoothness assumption}.

\begin{proof}[Proof of Theorem \ref{th: Hamiltonian smooth potential}]
    First assume that $u_0 \in \bigcap_{s \geq 0} \Hspacesp{s}{\RR^d}$, which ensures that we only manipulate strong derivatives.   \\
    We prove the following result by induction : for any $k \geq 0$ there exists a family of differential operators differential operator $P_k(t) = \sum_{\abs \alpha \leq k} a_\alpha^k(t, x) \partial_x^\alpha$ whose coefficients are smooth functions with bounded derivatives that only depend on the $\partial_t^j\partial_x^\alpha V$ for $j \leq k$ and $\abs \alpha \leq 2k$ such that
    \[
        \forall t \in \RR, \;
        (i\partial_t)^{k+1} v(t) = \eexp^{-it\laplacian} P_k(t) u(t),
    \]
    and there exists a constant $C_k$ that only depends on $k$ and $V$ such that
    \[
        \forall t \in \RR, \;
        \norm{u(t)}_{\Hspacesp{k}{\RR^d}}
        \leq    C_k (1 + \abs t^{\frac k 2}) \norm{u_0}_{\Hspacesp{k}{\RR^d}}.
    \]
    \begin{itemize}
        \item 
        For $k=0$ just take $P_0(t) = V(t)$, and use the fact that $\norm{u(t)}_{\Ltwosp{\RR^d}}$ is constant.
        
        \item
        Assume the result holds for $k-1$.
        Then
        \[
        \begin{aligned}
            (i\partial_t)^{k+1} v
            &=	(i\partial_t) (\eexp^{-it\laplacian} P_{k-1} \eexp^{it\laplacian} v)	\\
            &=	\eexp^{-it\laplacian} \Liebracket{\laplacian}{P_{k-1}} \eexp^{it\laplacian} v + \eexp^{-it\laplacian} (i\partial_t P_{k-1}) \eexp^{it\laplacian} v(t) + \eexp^{-it\laplacian} P_{k-1} \eexp^{it\laplacian} i\partial_t v	\\
            &=	\eexp^{-it\laplacian} \Liebracket{\laplacian}{P_{k-1}} u + \eexp^{-it\laplacian} (i\partial_t P_{k-1}) u + \eexp^{-it\laplacian} P_{k-1} V u,
        \end{aligned}
        \]
        and $P_k(t) = \Liebracket{\laplacian}{P_{k-1}(t)} + i\partial_t P_{k-1}(t) + P_{k-1}(t) V(t)$ is as stated.

        Now we compute
        \[
        \begin{aligned}
            &\frac{\diff{}}{\diff t} \dotprodbracket{(-\laplacian + V(t))^k u(t)}{u(t)}  \\
            &=  2 \underbrace{\Re \dotprodbracket{(-\laplacian + V(t))^k u(t)}{\partial_t u(t)}}_{= 0}
            + \Re \dotprodbracket{\sum_{j=1}^k (-\laplacian + V(t))^{j-1} \partial_t V(t) (-\laplacian + V(t))^{j-k} u(t)}{u(t)},
        \end{aligned}
        \]
        which implies
        \[
            \forall t \in \RR, \quad
            \frac{\diff{}}{\diff t}  \dotprodbracket{(-\laplacian + V(t))^k u(t)}{u(t)}
            \leq    C_k \norm{u(t)}_{\Hspacesp{k-1}{\RR^d}}^2
            \leq    C_k (1 + \abs t^{k-1}) \norm{u_0}_{\Hspacesp{k-1}{\RR^d}}^2.
        \]
        Integrating with respect to the time variable between $0$ and $t$, we obtain
        \[
            \forall t \in \RR, \quad
            \abs{\dotprodbracket{(-\laplacian + V(t))^k u(t)}{u(t)}}
            \leq    C_k (1 + \abs t^k) \norm{u_0}_{\Hspacesp{k}{\RR^d}}^2.
        \]
        The conclusion follows from \eqref{Smooth potential form domain equivalence}.
    \end{itemize}
    
    We extend the result to all $u_0 \in \Hspacesp{k}{\RR^d}$ by regularization: let $\chi \in \Ccinfsetsp{\RR^d}$ be such that $0 \leq \chi \leq 1$, and $\chi = 1$ near $0$ and for all $n\in \mathbb{N}$, set $u_0^n = \chi(\frac D  n) u_0 = \Fourier^{-1} \left( \chi(\frac \xi n) \Fourierhat{u_0} \right)$ (that is $\chi(\frac \cdot n)$ is used as a cut-off in the Fourier domain) and $u^n$ the solution to
    \[
    \begin{cases}
        i\partial_t u^n = (-\laplacian + V(t)) u^n,	\\
        u^n(0) = u_0^n,
    \end{cases}
    \]
    as well as $v^n = \eexp^{-it\laplacian} u^n$.
    Then
    \[
        u_0^n \xrightarrow[n \to \infty]{\Hspacesp{k}{\RR^d}} u_0,
    \]
    and
    \[
        \sup_{t \in \RR} \norm{v^*(t) - v^n(t)}_{\Ltwosp{\RR^d}}
        =	\norm{u_0 - u_0^n}_{\Ltwosp{\RR^d}}
        \xrightarrow[n \to \infty]{} 0,
    \]
    so in particular we have convergence in the weak sense. \\
    Additionally, it follows from what we proved earlier that $(i\partial_t)^{k+1} v^n$ is continuous with respect to $t$ and that for any $A > 0$,
    \[
    \forall m, n \in \mathbb{N}, \quad
    \norm{(i\partial_t)^{k+1} v^m - (i\partial_t)^{k+1} v^n}_{\Czerosetsp{[-A, A], \Ltwosp{\RR^d}}}
    \leq	C_k (1 + A^{\frac k 2}) \norm{u_0^m - u_0^n}_{\Hspacesp{k}{\RR^d}}.
    \]
    Therefore, $((i\partial_t)^{k+1} v^n)_n$ is a Cauchy sequence in the complete space $\Czerosetsp{[-A, A], \Ltwosp{\RR^d}}$, and there exists a $w \in \Czerosetsp{[-A, A], \Ltwosp{\RR^d}}$ such that
    \[
        (i\partial_t)^{k+1} v^n \xrightarrow[n \to \infty]{} w \quad \text{strongly in $\Czerosetsp{[-A, A], \Ltwosp{\RR^d}}$}.
    \]
    We infer, by identifying weak limits, that $(i\partial_t)^{k+1} v^* = w$, and
    \[
        \abs{(i\partial_t)^{k+1} v^*(t)}
        =	\lim_{n \to \infty} \abs{(i\partial_t)^{k+1} v^n(t)}
        \leq	C_k (1 + \abs t^{\frac k 2}) \lim_{n \to \infty} \norm{v_0^n}_{\Hspacesp{k}{\RR^d}}
        =	C_k (1 + \abs t^{\frac k 2}) \norm{v_0}_{\Hspacesp{k}{\RR^d}}
    \]
    for any $t \in [-A, A]$. Since $A$ is arbitrary, the result is proved.
\end{proof}
        
\subsection{Dual formulation} \label{subsec: Dual formulation}

The aim of this section is to comment on an alternative formulation and highlight the link between our work and the approach presented in~\cite{hain_ultra-weak_2024}.
    
In all this section we will assume, as in Theorem \ref{th: XX0 XX equality}, that the operators $H_0$ and $A$ satisfy the set of assumptions (A), $H = H_0 + A$, and $\fctshortdef{B}{t \in \setclosure I}{B(t)}$ is a strongly continuous family of bounded self-adjoint operators on $\HHilb$.

Using some ideas from~\cite{hain_ultra-weak_2024}, a different formulation on the less regular space $\Ltwosp{I, \HHilb}$ can be obtained. Let us first recall here the principle of the approach of~\cite{hain_ultra-weak_2024}. The idea of this method is to consider the Hilbert space
\begin{equation} \label{equ: YY definition}
    \YY_H = \set{u \in \XX_H}{u(T) = 0},
\end{equation}
equipped with the inner product
\begin{equation}
    \forall u_1, u_2 \in \YY_H, \quad
    \dotprodbracket{u_1}{u_2}_{\YY_H}
    =   \dotprodparenthesis{(i\partial_t - H) u_1}{(i\partial_t - H) u_2}_{\Ltwosp{I, \HHilb}}
\end{equation}
and the associated norm
\begin{equation} \label{equ: YY norm definition}
    \forall u \in \YY_H, \quad
    \norm u_{\YY_H} = \norm{(i\partial_t - H) u}_{\Ltwosp{I, \HHilb}}.
\end{equation}
Similarly, let us consider the space $\YY_{H_0}$ defined by \eqref{equ: YY definition} with $H = H_0$.

We then consider the antidual space
\begin{equation}
    \antidualspace{\YY_H}
    =   \set{\ell:\YY_H \to \CC}{\begin{array}{l} \text{$\ell$ is continuous and} \\ \forall u_1, u_2 \in \YY_H, \, \forall \lambda \in \CC, \, \ell(u_1 + \lambda u_2) = \ell(u_1) + \conjugate \lambda \ell(u_2) \end{array}},
\end{equation}
that is, the space of continuous antilinear forms on $\YY_H$.
We denote by $\dualitybracket \cdot \cdot_{\YY_H \times \antidualspace{\YY_H}}$ the antidual pairing $\YY_H \times \antidualspace \YY_H$, antilinear (resp. linear) with respect to the first (resp. second) variable. In particular, the Riesz representation theorem states that, for any $l \in \antidualspace \YY_H$, there exists a $u_l \in \YY_H$ such that
\[
    \forall u \in \YY_H, \quad
    \dualitybracket{u}{l}_{\YY_H \times \antidualspace{\YY_{H}}}
    =   \dotprodbracket{u}{u_l}_{\YY_H}.
\]

In~\cite{hain_ultra-weak_2024}, a variational formulation for (\ref{equ: Schrodinger evolution}) is proposed by means of a continuous bilinear form defined on $L^2(I; \HHilb \times \YY_H)$ which can be proved to satisfy classical inf-sup conditions. The conditioning of this formulation is proved to be optimal. However, one drawback is that there is no explicit caracterization of the set $\YY_H$. Our aim here is to provide an alternative formulation exploiting the fact that we can, in the present perturbative setting, give an explicit characterization of $\YY_H$. We then have the following result, which is proved analogously to Theorem \ref{th: XX0 XX equality}, simply by taking the integrals from $T$ to $t$ instead of $0$ to $t$. Let us point out that this result is an extension of Theorem~2.4 of~\cite{hain_ultra-weak_2024}, which enables us to treat unbounded potentials and Schrödinger equations defined on unbounded domains. 
\begin{proposition}
    Let $H_0$ and $A$ be operators satisfying the set of assumptions (A), and $\fctshortdef{B}{t \in \setclosure I}{B(t)}$ be a strongly continuous family of bounded self-adjoint operators on $\HHilb$. Set $H = H_0 + A$.
    It holds that $\YY_H = \YY_{H_0}$.
    Moreover, the map $L: \YY_{H_0} \to \RR_+$ defined by
    \begin{equation} \label{equ: L definition}
        \forall u \in \YY_{H_0}, \quad
        L(u) = \norm{(i\partial_t - H - B(t)) u}_{\Ltwosp{I, \HHilb}},
    \end{equation}
    defines a norm on $\YY_{H_0}$, and there exist constants $\alpha, C> 0$ independent of $T$ such that
    \begin{equation} \label{equ: XX0 XXt norm equivalence 2}
        \forall u \in \YY_{H_0}, \quad
        \frac{\alpha}{(1 + T)^2} \norm u_{\YY_{H_0}} \leq L(u) \leq C (1 + T) \norm u_{\YY_{H_0}}.
    \end{equation}
\end{proposition}
 Now consider the operator $\fctshortdef{\adjoint S}{\Ltwosp{I, \HHilb}}{\antidualspace \YY_{H_0}}$ defined by
\begin{equation}
\label{equ: S-star definition}
    \forall w \in \Ltwosp{I, \HHilb}, \quad \forall u \in \YY_{H_0} \quad
    \dualitybracket{u}{\adjoint S w}_{\YY_{H_0} \times \antidualspace{\YY_{H_0}}}
    =	\dotprodparenthesis{(i\partial_t - H - B(t)) u}{w}_{\Ltwosp{I, \HHilb}},
\end{equation}
and similarly the operator $\fctshortdef{\adjoint S_0}{\Ltwosp{I, \HHilb}}{\antidualspace \YY_{H_0}}$ defined by 
    \begin{equation}
\label{equ: S_0-star definition}
    \forall w \in \Ltwosp{I, \HHilb}, \quad \forall u \in \YY_{H_0} \quad
    \dualitybracket{\adjoint S_0 w}{u}_{\YY_{H_0} \times \antidualspace{\YY_{H_0}}}
    =	\dotprodparenthesis{w}{(i\partial_t - H_0) u}_{\Ltwosp{I, \HHilb}}.
\end{equation}

\begin{lemma}
    We have the following estimates :
    \begin{equation} \label{equ: S-star coercivity estimates}
        \forall w \in \Ltwosp{I, \HHilb}, \quad
        \frac{\alpha}{(1+T)^2} \norm w_{\Ltwosp{I, \HHilb}}
        \leq	\norm{S^* w}_{\YY_0^\dagger}
        \leq	C(1+T) \norm w_{\Ltwosp{I, \HHilb}}.
    \end{equation}
\end{lemma}

\begin{proof}
    Both estimates follow from \eqref{equ: XX0 XXt norm equivalence 2} by a classical inf-sup argument which we recall here.
    
    For all $w \in \Ltwosp{I, \HHilb}$, one has with \eqref{equ: L definition}
    \[
    \begin{aligned}
        \norm{S^* w}_{\YY_0^\dagger}
        &=	\sup_{\norm u_{\YY_{H_0}} = 1} \dualitybracket{(i\partial_t - H - B(t)) u}{w}_{\YY_{H_0} \times \antidualspace{\YY_{H_0}}}
        \leq	\sup_{\norm u_{\YY_{H_0}} = 1} \norm w_{\Ltwosp{I, \HHilb}} L(u)   \\
        &\leq	C(1+T) \norm w_{\Ltwosp{I, \HHilb}},
    \end{aligned}
    \]
    which proves the second inequality in \eqref{equ: S-star coercivity estimates}.

    For the first inequality, define $u_w$ as the unique element of $\YY_{H_0}$ such that $(i\partial_t - H - B(t)) u_w = w$ (in particular $L(u_w) = \norm w_{\Ltwosp{I, \HHilb}}$), then by \eqref{equ: XX0 XXt norm equivalence 2}
    \[
    \begin{aligned}            
        \norm{S^* w}_{\antidualspace{\YY_{H_0}}}
        &\geq	\frac{\dualitybracket{w}{(i\partial_t - H - B(t)) u_w}_{\YY_{H_0} \times \antidualspace{\YY_{H_0}}}}{\norm{u_w}_{\YY_{H_0}}}
        =	\frac{\norm w_{\Ltwosp{I, \HHilb}}^2}{\norm{u_w}_{\YY_{H_0}}}
        =   \norm w_{\Ltwosp{I, \HHilb}} \frac{L(u_w)}{\norm{u_w}_{\YY_{H_0}}}   \\
        &\geq	\frac{\alpha}{(1+T)^2} \norm w_{\Ltwosp{I, \HHilb}},
    \end{aligned}
    \]
    which yields the desired result.
\end{proof}

For $u_0 \in \HHilb$ and $f \in \Ltwosp{I, \HHilb}$, we denote by $i \delta_{t=0} \otimes u_0 + f$ the element of $\antidualspace{\YY_{H_0}}$ defined by
\begin{equation}
    \forall u \in \YY_{H_0}, \quad
    \dualitybracket{u}{i \delta_{t=0} \otimes u_0 + f}_{\YY_{H_0} \times \antidualspace{\YY_{H_0}}}
    =   i \dotprodbracket{u(0)}{u_0} + \dotprodparenthesis{u}{f}_{\Ltwosp{I, \HHilb}}
\end{equation}
We have the following proposition which connects the Schrödinger evolution equation and the operators $\adjoint S$ and $\adjoint S_0$.
\begin{proposition} \label{prop: S-star connection with the PDE}
\begin{itemize}
    \item[(i)] The function $u \in \Ltwosp{I, \HHilb}$ solves \eqref{equ: Schrodinger evolution} in the sense of Definition~\ref{def: definition weak solution} if and only if
    \[
        \adjoint S u = (\adjoint S_0 + A + B(t)) u
            = i \delta_0 \otimes u_0 + f,
    \]
    or equivalently
    \[
        ({\rm Id} + (\adjoint S_0)^{-1} (A + B(t))) u
            = (\adjoint S_0)^{-1}(i \delta_0 \otimes u_0 + f).
    \]
    \item[(ii)] For any $u_0 \in \HHilb$ and $f \in \Ltwosp{I, \HHilb}$, we have
    \begin{equation}
        (\adjoint S_0)^{-1}(i \delta_0 \otimes u_0 + f)(t, x)
        =	\eexp^{-itH_0} u_0 - i \int_0^t \diff s \eexp^{-i(t-s)H_0} f(s).
    \end{equation}
\end{itemize}
\end{proposition}

\begin{proof}
    Assume $u$ solves \eqref{equ: Schrodinger evolution}. Then $u \in \XX_H$ and we have for any $v \in \YY_{H_0}$ (using the ``integration by part'' formula \eqref{equ: XX IBP})
    \[
    \begin{aligned}
        \dualitybracket{v}{\adjoint S u}_{\YY_{H_0} \times \antidualspace{\YY_{H_0}}}
        &=	\dotprodbracket{(i\partial_t - H - B(t))v}{u}_{\Ltwosp{I, \HHilb}}
        =	\int_I \diff t \dotprodbracket{(i \partial_t - H - B(t)) v}{u}	\\
        &=	\dotprodbracket{v(0)}{i \underbrace{u(0)}_{=u_0}} + \dotprodbracket{v}{\underbrace{(i\partial_t - H - B(t)) u}_{= f}}_{\Ltwosp{I, \HHilb}}	\\
        &=	\dualitybracket{v}{i \delta_{t=0} \otimes u_0 + f}_{\YY_{H_0} \times \antidualspace{\YY_{H_0}}}.
    \end{aligned}
    \]
    
    Conversely, assume $\adjoint S u = i \delta_{t=0} \otimes u_0 + f$. Then, for any $w \in \Ccinfsetsp{I, \HHilb} \subset \YY_{H_0}$,
    \[
    \begin{aligned}
        \dotprodparenthesis{i\partial_t w}{\eexp^{itH}u}_{\Ltwosp{I, \HHilb}}
        &=	\dotprodparenthesis{\eexp^{-itH}i\partial_t w}{u}_{\Ltwosp{I, \HHilb}}
        =	\dotprodparenthesis{(i\partial_t - H) \eexp^{-itH} w}{u}_{\Ltwosp{I, \HHilb}}	\\
        &=	\dualitybracket{\eexp^{-itH} w}{\adjoint S u}_{\YY_{H_0} \times \antidualspace{\YY_{H_0}}}
        =	\dotprodbracket{\underbrace{w(0)}_{=0}}{i u_0} + \dotprodparenthesis{\eexp^{-itH} w}{f}_{\Ltwosp{I, \HHilb}}	\\
        &=	\dotprodparenthesis{w}{\eexp^{itH} f}_{\Ltwosp{I, \HHilb}}.
    \end{aligned}
    \]
    This means that $\eexp^{itH} u \in \Honesp{I, \HHilb}$ with $i\partial_t \eexp^{itH} u = \eexp^{itH} f$, hence $u \in \XX_H$ and $(i\partial_t - H) u = f$.
    Therefore, for any $v \in \YY_{H_0}$,
    \[
    \begin{aligned}
        \dotprodparenthesis{(i\partial_t - H) v}{u}_{\Ltwosp{I, \HHilb}}
        =
        \begin{cases}
            \dotprodbracket{v(0)}{i u_0} + \dotprodparenthesis{v}{f}_{\Ltwosp{I, \HHilb}},	\\
            \dotprodbracket{v(0)}{i u(0)} + \dotprodparenthesis{v}{(i\partial_t - H) u}_{\Ltwosp{I, \HHilb}} = \dotprodbracket{v(0)}{i u(0)} + \dotprodparenthesis{v}{f}_{\Ltwosp{I, \HHilb}}.
        \end{cases}
    \end{aligned}
    \]
    It follows that
    \[
        \forall v \in \YY, \quad
        \dotprodbracket{i u(0)}{v(0)} = \dotprodbracket{i u_0}{v(0)},
    \]
    and finally
    \[
        u(0) = u_0.
    \]
    
    The second part of the proposition is just the consequence of the first with $A=0$ and $B(t) = 0$ for all $t \in \setclosure I$.
\end{proof}

The following corollary gives an alternative variational formulation for the solution to the time-dependent Schrödinger equation.
\begin{corollary}\label{cor:L2}
    Let $u_0 \in \HHilb$, $f \in \Ltwosp{I, \HHilb}$.
    The functional
    \begin{equation}
        E^*(u)
        =	\norm{u(t) + i \int_0^t \diff s \eexp^{-i(t-s)H_0}(Au) - \left( \eexp^{-itH_0} u_0 - i \int_0^t \diff s \eexp^{-i(t-s)H_0} f(s) \right)}_{\Ltwosp{I, \HHilb}}^2
    \end{equation}
    is a strongly convex quadratic functional defined on $\Ltwosp{I, \HHilb}$ which satisfies
    \begin{equation}
        \forall u \in \Ltwosp{I, \HHilb}, \quad
        \frac{\alpha}{(1+T)^2} \norm{u - u^*}_{\Ltwosp{I, \HHilb}} \leq \sqrt{E^*(u)} \leq C(1+T) \norm{u - u^*}_{\Ltwosp{I, \HHilb}}
    \end{equation}
    where $u^*$ is the solution to \eqref{equ: Schrodinger evolution} in the sense of Definition \ref{def: definition weak solution}.
\end{corollary}

\begin{proof}[Proof of Corollary~\ref{cor:L2}]
    Let $u^*$ denote the solution to \eqref{equ: Schrodinger evolution}.
    It follows from Proposition \ref{prop: S-star connection with the PDE} that
    \[
        \eexp^{-itH_0} u_0 - i \int_0^t \diff s \eexp^{-i(t-s)H_0} f(s)
        =	(S_0^*)^{-1}(\delta_0 \otimes u_0 + f)
        =	(S_0^*)^{-1} S^* u^*,
    \]
    and
    \[
        u + i \int_0^t \diff s \eexp^{-i(t-s)H_0}(Au)
        =	u - (S_0^*)^{-1} Au
        =	(S_0^*)^{-1} S^* u.
    \]
    Therefore, since clearly $S_0^*$ is an isometry between $\YY_{H_0}$ and $\antidualspace{\YY_{H_0}}$,
    \[
        E^*(u)
        =	\norm{(S_0^*)^{-1} S^*(u - u^*)}_{\Ltwosp{I, \HHilb}}^2
        =	\norm{S^*(u - u^*)}_{\antidualspace{\YY_{H_0}}}^2,
    \]
    and the desired estimate is a consequence of \eqref{equ: S-star coercivity estimates}.
\end{proof}

\section{Conclusion and perspectives }\label{sec:low-rank}

In this work, we establish a space-time variational formulation for the Schrödinger evolution equation~\eqref{equ: Schrodinger evolution}. 
This formulation encompasses the case of many-body electronic Coulomb interaction, or more general charge densities (see Remark~\ref{rem: general density}), with a bounded time-dependent potential.

\medskip

As mentioned in the introduction, this variational formulation was mainly developed to provide a numerically achievable way to compute low complexity approximations of the solution to \eqref{equ: Schrodinger evolution}.
This will be the topic of a future work. We describe the main ideas here.
One key ingredient is the following proposition:
\begin{proposition}\label{prop:sigma}
    Let $\Sigma$ be any non empty weakly closed subset of $\HHilb$. Then the subset
    \[
        \Honesp{I, \Sigma} = \set{u \in \Honesp{I, \HHilb}}{u(t) \in \Sigma \text{ for any $t \in I$}}
    \]
    is closed in $\Honesp{I, \HHilb}$ for the weak topology.
\end{proposition}

\begin{proof}
    Since $\Honesp{I, \HHilb} \hookrightarrow \Cregsetsp{0}{I, \HHilb}$, for any $t \in I$ the linear application
    \[
        \fctlongdef{T_t}{\Honesp{I, \HHilb}}{\HHilb}{v}{v(t)}
    \]
    is continuous. \\
    Let $(v_n)_{n \geq 0} \subset \Honesp{I, \Sigma}$ be a sequence which weakly converges to some $v$ in $\Honesp{I, \HHilb}$. Since any bounded linear operator between Banach spaces is also weakly continuous, this implies that for any $t \in I$,
    \[
        v_n(t) = T_t v_n \mathop{\rightharpoonup}_{n\to +\infty} T_t v = v(t) \quad
        \text{in $\HHilb$},
    \]
    and the weak closedness of $\Sigma$ implies $v(t) \in \Sigma$.
\end{proof}

One possible interesting choice for the set $\Sigma$ is some given set of functions which can be represented with low complexity, such as tensor formats, for instance the manifold of tensor trains with at most a given rank, or well-chosen neural network architectures. Proposition~\ref{prop:sigma} implies that, for any value of the final time $T>0$, there exists at least one minimizer to the minimization problem
\begin{equation}\label{eq:newVP}
u_\Sigma^*  \in  \argmin_{v\in \Honesp{I, \Sigma}} F(v)
\end{equation}
where $F$ is defined in \eqref{equ: F definition}.

As a consequence, for instance when $\Sigma$ is given as some low-rank tensor format, the variational principle studied here enables to obtain the existence of a dynamical low-rank approximation of $u^*$ whatever the value of the final time $T$. This approximation is thus obtained through the variational principle~(\ref{eq:newVP}) which is different from the classical Dirac-Frenkel one~\cite{Conte_Lubich_2010}, for which, at least up to our knowledge, only local-in-time existence of solutions has been proved in the general case.
However, we emphasize here that, while some quantities are preserved by the exact Schrödinger dynamics \eqref{equ: Schrodinger evolution}, these same quantities are not necessarily preserved for the solution to \eqref{eq:newVP}.
The comparison of both types of dynamical low-rank approximations will be the object of a forthcoming article.

The alternative variational formulation presented in Section~\ref{subsec: Dual formulation} also exhibits some interesting properties for dynamical low-rank approximations.
Notably, it allows discontinuous functions in time, which is useful since the best low-rank tensor approximation of $u^*(t)$ is not always continuous in $t$.
However, the existence of a global-in-time low-rank dynamical approximation of $u^*$ using this formulation is not guaranteed in general, in contrast to the formulation presented in Section~\ref{subsec: Abstract setting}.
More importantly, from a practical standpoint, it typically results in higher computational costs due to the presence of a time integral.

\section*{Acknowledgements}

This publication is part of a project that has received funding from the European Research Council (ERC) under the European Union’s Horizon 2020 Research and Innovation Programme – Grant Agreement
n$^\circ$ 101077204. We are also very thankful to Caroline Lasser, Christian Lubich and Elena Celledoni for enlightening discussions on this work.
We thank the language model \textit{Le Chat} for its assistance in refining formulations during the writing process.

\appendix

\section{Numerical implementation with spectral methods} \label{sec: Numerical implementation}

    The aim of this section is to illustrate the interest of the proposed variational formulation for numerical purposes on some simple test cases. While our main motivation stems from the design of dynamical low-rank approximation schemes, we illustrate here the behaviour of Galerkin global-space time discretization schemes stemming from the variational formulation presented here. The latter mainly rely on the results of Section~\ref{sec: Space-time variational formulation for the Schrödinger equation}, and in particular Corollary \ref{cor: Variational formulation v}.
    The code can be found at \cite{guillot_2024_11353969}.
    
    In Section~\ref{subsec: Time regularity of the solution}, we prove a regularity result of the solutions with respect to the time variable for smooth interaction time-dependent potentials, which motivates the use of spectral methods in this simple case.
    We take advantage of this additional smoothness, coupled with the global space-time formulation obtained in Corollary \ref{cor: Variational formulation v}, to apply spectral methods to the time variable.

    In Section~\ref{subsec: An ordinary differential equation}, we begin with an elementary ordinary differential equation. Although simple, this first example is an opportunity to present the basic principles of the spectral discretization for the time variable in detail. We also give an example of a simple but efficient preconditioner which will also be useful for the next example.

    In Section~\ref{subsec: Periodic Schrödinger equation}, we present the numerical behaviour of the scheme applied to the simulation of the evolution of a 2D Schrödinger equation with periodic boundary conditions.

\subsection{A first basic example} \label{subsec: An ordinary differential equation}

Consider the ordinary differential equation ($a, \omega \in \RR$)
\begin{equation}
\begin{cases}
    i(u^*)'(t) = a (\cos \omega t) u^*(t),	\\
    u^*(0) = \eta_0 \in \CC.
\end{cases}
\end{equation}
The solution admits the explicit expression $u^*(t) = z_0 \eexp^{-i a \frac{\sin \omega t}{\omega}}$.

We will compute an approximation of this solution on the time interval $(-1,1)$.
We claim that the results established in the paper for $(0,T)$ extend straightforwardly to $(-T,T)$, due to the unitary nature of $\eexp^{-itH}$ when $H$ is a self-adjoint operator.
Corollary \ref{cor: Variational formulation v} (with $H_0 = A = 0$, and $B(t) = a (\cos \omega t)$) leads to the following problem:
\begin{equation}
    \min_{u \in \Hone(-1, 1)} \left( \abs{u(0) - \eta_0}^2 + \int_{-1}^1 \diff t \abs{iu'(t) - a (\cos \omega t) u(t)}^2 \right).
\end{equation}
However, in order to make the numerical computations easier, we will consider instead the minimization problem
\begin{equation} \label{equ: ODE minimization problem}
    \min_{u \in \Hone_w(-1, 1)} E_w(u), \quad
    \text{with } E_w(u) := \abs{u(0) - \eta_0}^2 + \int_{-1}^1 \diff t w(t) \abs{iu'(t) - a (\cos \omega t) u(t)}^2,
\end{equation}
where $w(t) = \sqrt{1 - t^2}$, and $u \in \Hone_w(-1, 1)$ means that $\int_{-1}^1 \diff t w(t)( \abs{u(t)}^2 + \abs{u'(t)}^2 )< \infty$.
For convenience, we extended the problem to the whole time interval $(-1, 1)$, since it is the natural domain for the Chebyshev polynomials. The alternative would be to rescale the problem posed on $(0, 1)$ to obtain a new problem posed on $(-1, 1)$.
In this case, the error estimate
\[
    \norm{u - u^*}_{\Czerosetsp{(-1, 1)}}
    \leq	\sqrt{2 E_w(u)}
\]
does not hold on $(-1, 1)$, but it can be replaced in this case by the similar estimate
\begin{equation}
    \norm{u - u^*}_{\Czerosetsp{(-1, 1)}}
    \leq    \sqrt{2 \pi E_w(u)}
\end{equation}
obtained as follows : for any $u \in \Hone_w(-1, 1)$, and any $0 \leq t < 1$,
\[
\begin{aligned}
    \abs{u(t)}^2
    &=    \abs{u(0)}^2 + 2 \int_0^t \diff s \Im \left( \conjugate{(iu'(s) - a(\cos \omega s) u(s))} u(s) \right)   \\
    &\leq   \abs{u(0)}^2 + 2 \int_0^t \diff s \abs{iu'(s) - a(\cos \omega s) u(s)} \abs{u(s)}   \\
    &\leq   \abs{u(0)}^2 + 2 (\sup_{0 \leq s \leq t} \abs{u(s)}) \left( \int_0^1 \frac{\diff s}{w(s)} \right)^{\frac 1 2} \left( \int_0^t \diff s w(s) \abs{iu'(s) - a(\cos \omega s) u(s)}^2 \right)^{\frac 1 2}   \\
    &\leq   \abs{u(0)}^2 + \frac 1 2 \sup_{0 \leq s \leq t} \abs{u(s)}^2 + 2 \underbrace{\left( \int_0^1 \frac{\diff s}{w(s)} \right)}_{= \frac \pi 2} \left( \int_0^t \diff s w(s) \abs{iu'(s) - a(\cos \omega s) u(s)}^2 \right).
\end{aligned}
\]
Since a similar estimate can be obtained for $-1 < t \leq 0$, the result follows.

We introduce $(T_k)_{k \geq 0}$ and $(U_k)_{k \geq 0}$ the Chebyshev polynomials of the first and second kind respectively. We recall in particular the orthogonality relation:
\begin{equation} \label{equ: Chebyshev prop orthogonality type 2}
    \forall k, l,
    \quad \int_{-1}^1 \diff t w(t) U_k(t) U_{\ell}(t) = \delta_{k\ell} \frac{\pi}{2}
\end{equation}

Fix some $K \geq 0$ and consider the discrete space
\[
    X^K
    =	\set{u(t) = \sum_{k=0}^{K-1} u_k T_k(t)}{\boldsymbol{u}:=(u_k)_{k=0}^{K-1} \subset \CC}
    \subset \Hone_w(-1, 1).
\]
For any $u\in X^K$ and associated coordinates $\boldsymbol{u} = (u_k)_{k=0}^{K-1}\in \CC^K$, we wish to compute $u(0)$ and an approximation of $iu'(t) - a (\cos \omega t) u(t)$ of the form $g(t) = \sum_{\ell=0}^{L-1} g_\ell U_\ell(t) $ for some $L \geq K$ and coefficients $(w_k)_{k=0}^{L-1}\in \CC^L$. 
\begin{itemize}
    \item For $u(0)$, we can simply write
    \[
        u(0) = \sum_{k=0}^{K-1} u_k T_k(0),
    \]
    which motivates the introduction of the vector
    \begin{equation}
        \boldsymbol{\mathcal J}^K
        =
        \begin{pmatrix}
            T_0(0) & T_1(0) & \ldots & T_{K-1}(0)
        \end{pmatrix}
        \in	\RR^{1 \times K},
    \end{equation}
    so that $u(0) = \boldsymbol{\mathcal J}^K \boldsymbol{u}$.
    \item Define
    \begin{equation}
        \boldsymbol{\mathcal P}^{L, K}
        =
        \begin{pmatrix}
            I_K \\ 0_{(L-K) \times K}
        \end{pmatrix}
    \end{equation}
    the extension operator where $I_K$ is the $K \times K$ identity matrix, and $0_{(L-K) \times K}$ the $(L-K) \times K$ null matrix.
    
    \item The term $iu'(t)$ is easy to compute because of the identity $\forall k, \, T_{k+1}'= (k+1)U_k$. We define the corresponding matrix
    \begin{equation}
       \boldsymbol{\mathcal D}^K
        =
        \begin{pmatrix}
            0	&	i	&	0	&	\ldots	&	0	\\
            0	&	0	&	2i	&	\ldots	&	0	\\
            \vdots	&	\vdots	&	\vdots	&	\ddots	&	\vdots	\\
            0	&	0	&	0	&	\ldots	&	(K-1)i	\\
            0	&	0	&	0	&	\ldots	&	0	\\
        \end{pmatrix}.
    \end{equation}
    
    \item To approximate the term $a (\cos \omega t) u(t)$, we rely on the following method : let $\fctshortdef{\boldsymbol \pi^L}{\CC^L}{\CC^L}$ be the ``collocation operator with L points'', that is,
    \[
        (\boldsymbol \pi^L \mathbf u)_{\ell}
        =	\sum_{k=0}^{L-1} u_k T_k(x_{\ell}),
    \]
    where $(x_\ell)_{\ell=0}^{L-1}$ are the Gauss-Chebyshev nodes.
    We know that there exist positive weights $(\mu_{\ell})_{l=0}^{L-1}$ such that the quadrature formula
    \[
        \int_{-1}^1 \frac{\diff t}{w(t)} P(t)
        \approx	\sum_{l=0}^{L-1} \mu_{\ell} P(x_{\ell}),
    \]
    is exact whenever $P$ is a polynomial with degree $\leq 2L-1$. \\
    Let also $\boldsymbol{\mathcal A}^L$ be the operator $\CC^L \to \CC^L$ such that
    \[
        \forall \boldsymbol z := (z_{\ell})_{l=0}^{L-1}, \quad
        (\boldsymbol{\mathcal A^L} \boldsymbol z)_{\ell} = a (\cos \omega x_{\ell}) z_{\ell}
    \]
    The last operation required is the transformation from the first-kind series to the second-kind series, which can be expressed in coordinates as follows thanks to the identities $\forall k \geq 2, \,T_k = \frac{U_k - U_{k-2}}{2}$, $T_1 = \frac{U_1}{2}$, and $T_0 = U_0$:
    \begin{equation}
       \boldsymbol{\mathcal C}^L
        =
        \begin{pmatrix}
            1	&	0	&	- \frac 1 2	&	\ldots	&	\ldots	&	\ldots	&	0	\\
            0	&	\frac 1 2	&	0	&	\ddots	&		&		&	0	\\
            0	&	0	&	\frac 1 2	&	\ddots	&	\ddots	&		&	0	\\
            \vdots	&	\vdots	&	\vdots	&	\ddots	&	\ddots	&	\ddots	&	\vdots	\\
            0	&	0	&	0	&		&	\ddots	&	\ddots	&	-\frac 1 2	\\
            0	&	0	&	0	&		&		&	\ddots	&	0	\\
            0	&	0	&	0	&	\ldots	&	\ldots	&	\ldots	&	\frac 1 2	\\
        \end{pmatrix}.
    \end{equation}
\end{itemize}

The matrix of the quadratic part of \eqref{equ: ODE minimization problem} can therefore be written
\begin{equation} \label{equ: Discrete Chebyshev quadratic form}
\begin{aligned}
    \boldsymbol{\mathcal Q}_{K, L}
    &=
    \frac \pi 2 \adjoint{\left( \boldsymbol{\mathcal P}^{L, K} \boldsymbol{\mathcal D}^K - \boldsymbol{\mathcal C}^L (\boldsymbol \pi^L)^{-1} \boldsymbol{\mathcal A}^L \boldsymbol \pi^L \boldsymbol{\mathcal P}^{K, L} \right)} \left( \boldsymbol{\mathcal P}^{L, K} \boldsymbol{\mathcal D}^K - \boldsymbol{\mathcal C}^L (\boldsymbol \pi^L)^{-1} \boldsymbol{\mathcal A}^L \boldsymbol \pi^L \boldsymbol{\mathcal P}^{K, L} \right)   \\
    &\quad \quad +	\adjoint{(\boldsymbol{\mathcal J}^K)} \boldsymbol{\mathcal J}^K
\end{aligned}
\end{equation}

The only non explicit part of the above formula is the adjoint of $(\boldsymbol \pi^L)^{-1} \boldsymbol{\mathcal A}^L \boldsymbol \pi^L$, which we compute as follows :
Let $\boldsymbol C = diag(\frac 1 2, 1 ..., 1) \in \RR^{L \times L}$.
Fix $\boldsymbol u = (u_\ell)_{\ell=0}^{L-1}$ and $\boldsymbol v = (v_\ell)_{\ell=0}^{L-1}$.
Then
\[
\begin{aligned}
    \sum_{\ell=0}^{L-1} \conjugate{((\boldsymbol \pi^L)^{-1} \boldsymbol{\mathcal A}^L \boldsymbol \pi^L u)_\ell} v_\ell
    &=	\frac 2 \pi \int \frac{\diff t}{w(t)} \conjugate{\left( \sum_{\ell=0}^{L-1} ((\boldsymbol \pi^L)^{-1} \boldsymbol{\mathcal A}^L \boldsymbol \pi^L \boldsymbol u)_\ell T_l(t) \right)} \left( \sum_{\ell=0}^{L-1} (\boldsymbol C \boldsymbol v)_\ell T_l(t) \right)   \\
    &=	\frac 2 \pi \sum_{l=0}^{L-1} \mu_{\ell} \conjugate{(\boldsymbol{\mathcal A} \boldsymbol \pi^L \boldsymbol u)_{\ell}} (\boldsymbol \pi^L \boldsymbol C \boldsymbol v)_{\ell}
    =	\frac 2 \pi \sum_{\ell=0}^{L-1} \mu_{\ell} a (\cos \omega x_{\ell}) \conjugate{(\boldsymbol \pi^L \boldsymbol u)_{\ell}} (\boldsymbol \pi^L \boldsymbol C \boldsymbol v)_{\ell}   \\
    &=	\frac 2 \pi \sum_{\ell=0}^{L-1} \mu_{\ell} \conjugate{(\boldsymbol \pi^L \boldsymbol u)_{\ell}} (\boldsymbol{\mathcal A} \boldsymbol \pi^L \boldsymbol C \boldsymbol v T_k)_{\ell}   \\
    &=	\frac 2 \pi \int \frac{\diff t}{w(t)} \conjugate{\left( \sum_{\ell=0}^{L-1} u_\ell T_l(t) \right)} \left( \sum_{\ell=0}^{L-1} ((\boldsymbol \pi^L)^{-1} \boldsymbol{\mathcal A}^L \boldsymbol \pi^L \boldsymbol C \boldsymbol v)_\ell T_l(t) \right)   \\
    &=	\sum_{\ell=0}^{L-1} \conjugate{u_\ell} (\boldsymbol C^{-1} (\boldsymbol \pi^L)^{-1} \boldsymbol{\mathcal A} \boldsymbol \pi^L \boldsymbol C \boldsymbol v)_\ell,
\end{aligned}
\]
and the desired transpose is therefore $\boldsymbol C^{-1} (\boldsymbol \pi^L)^{-1} \boldsymbol{\mathcal A} \boldsymbol \pi^L \boldsymbol C$.
We therefore see that the minimization problem \eqref{equ: ODE minimization problem} is equivalent to the linear system
\begin{equation} \label{equ: Discrete Chebyshev linear system}
    \boldsymbol{\mathcal Q}_{K, L} \boldsymbol u^*_{K, L} = \eta_0 \adjoint{(\boldsymbol{\mathcal J}^K)}.
\end{equation}
One may observe that computing $\boldsymbol{\mathcal Q}_{K, L} \boldsymbol u$ using \eqref{equ: Discrete Chebyshev quadratic form} only requires $\bigO{L \log L}$ operations, it is therefore tempting to try to solve \eqref{equ: Discrete Chebyshev linear system} with an iterative method such as the conjugate gradient. 
However, it turns out that the matrix $\boldsymbol{\mathcal Q}_{K, L}$ is in fact ill-conditioned, and the amount of iterations required to converge increases drastically with the size of $K$. This problem can be solved by introducing the matrix
\begin{equation}
    \boldsymbol Q_K^0
    =   \frac \pi 2 \adjoint{\left(\boldsymbol{\mathcal D}^K \right)} \boldsymbol{\mathcal D}^K
    +	\adjoint{(\boldsymbol{\mathcal J}^K)} \boldsymbol{\mathcal J}^K,
\end{equation}
which is simply the matrix associated with the ``free'' quadratic form
\[
    \abs{u(0)}^2 + \int_{-1}^1 \diff t w(t) \abs{i\partial_t u}^2.
\]
It turns out that the inverse of $\boldsymbol{\mathcal Q}_0$ can be easily computed as follows : let $F \in \RR^K$, then $\boldsymbol{\mathcal Q}_0 \boldsymbol u = F$ is equivalent to finding $u = \sum_{k=0}^K u_k T_k$ such that
\[
    \forall v = \sum_{k=0}^K v_k T_k, \quad
    \int_{-1}^1 \diff t w(t) (\conjugate{i\partial_t v})(i\partial_t u) + \conjugate{v(0)}u(0)
    =	\dotprodbracket{v}{F}_{\RR^{K+1}}
    =	\sum_{k=0}^K \conjugate{v_k} F_k.
\]
Taking $v = T_0$ yields $u(0) = F_0$, and then we obtain
\[
    \forall v = \sum_{k=0}^K k^2 v_k T_k, \quad
    \frac \pi 2 \sum_{k=1}^K \conjugate{v_k} u_k
    =	\sum_{k=0}^K \conjugate{v_k} F_k - \left( \sum_{k=0}^K \conjugate{v_k} T_k(0) \right) F_0,
\]
so we need to take
\[
    \forall 1 \leq k \leq K, \quad
    u_k
    =	\frac{2}{\pi k^2} (F_k - T_k(0) F_0).
\]
We then find $u_0$ thanks to the following equality :
\[
    F_0 = u(0) = \sum_{k=0}^K T_k(0) u_k.
\]

Figure \ref{fig: EDO error aliasing} shows $\norm{u^*_{K, L} - u^*}_{\Czerosetsp{(0, 1)}}$, computed by sampling at $1000$ uniformly distributed points as a function of $K$ for different choices of $L$. Since taking $L > K$ does not seem to significatively improve the convergence, in what follows we always take $L = K$ and write $u^*_K := u^*_{K, K}$.

On Figure \ref{fig: EDO error plot}, we display the difference with the exact solution computed at $1000$ uniformly distributed points on the interval $[0, 1]$, and compare it with an order 2 Crank-Nicolson scheme as a function of $K$. For the the Crank-Nicolson scheme, $K$ stands for the number of steps. As expected, the global in time Chebyshev approach exhibits spectral convergence, while the error evolves as $K^{-2}$ (convergence of order 2) for the Crank-Nicolson scheme.
A Crank-Nicolson scheme with $K$ steps is cheaper than the procedure described above with $K$ functions, therefore, a higher value of $K$ does not necessarily mean that the computation is more expensive. The evolution of the computation time will be explored in the next example.

\begin{figure}[h] \label{fig:Gaussian3D simulation}
	\centering
	% First subfigure
	\begin{subfigure}[t]{0.49\textwidth}
		\centering
		\includegraphics[width=\textwidth]{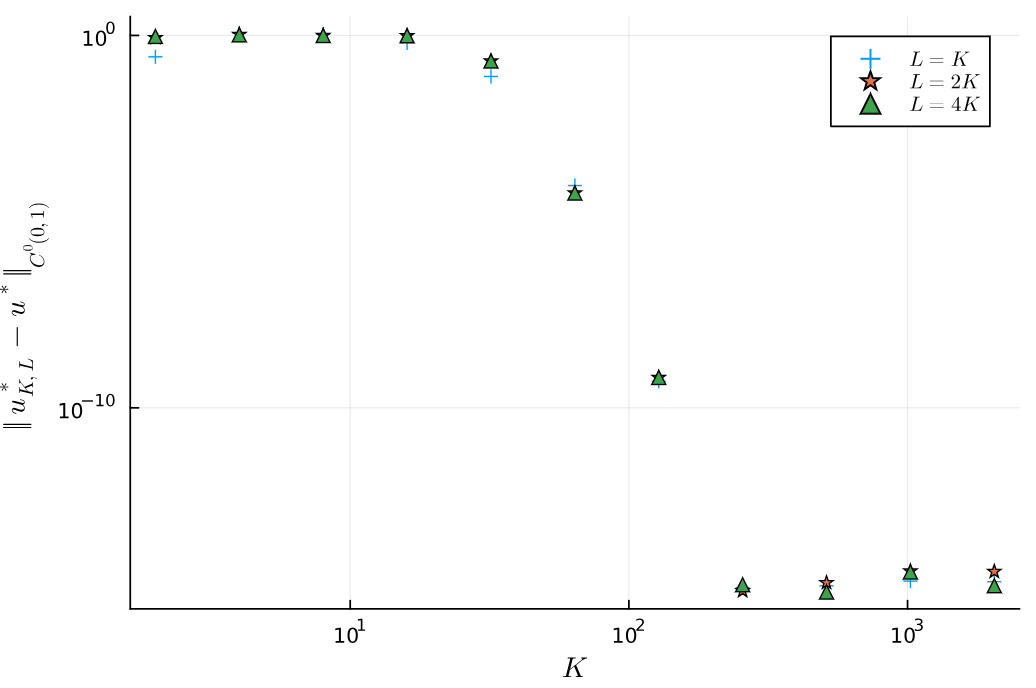}
		\caption{$\norm{u^*_{K, L} - u^*}_{\Czerosetsp{(-1, 1)}}$ for different choices of $L$ with $a=5$ and $\omega=20$}
            \label{fig: EDO error aliasing}
	\end{subfigure}
	\hfill
	% Second subfigure
	\begin{subfigure}[t]{0.49\textwidth}
		\centering
		\includegraphics[width=\textwidth]{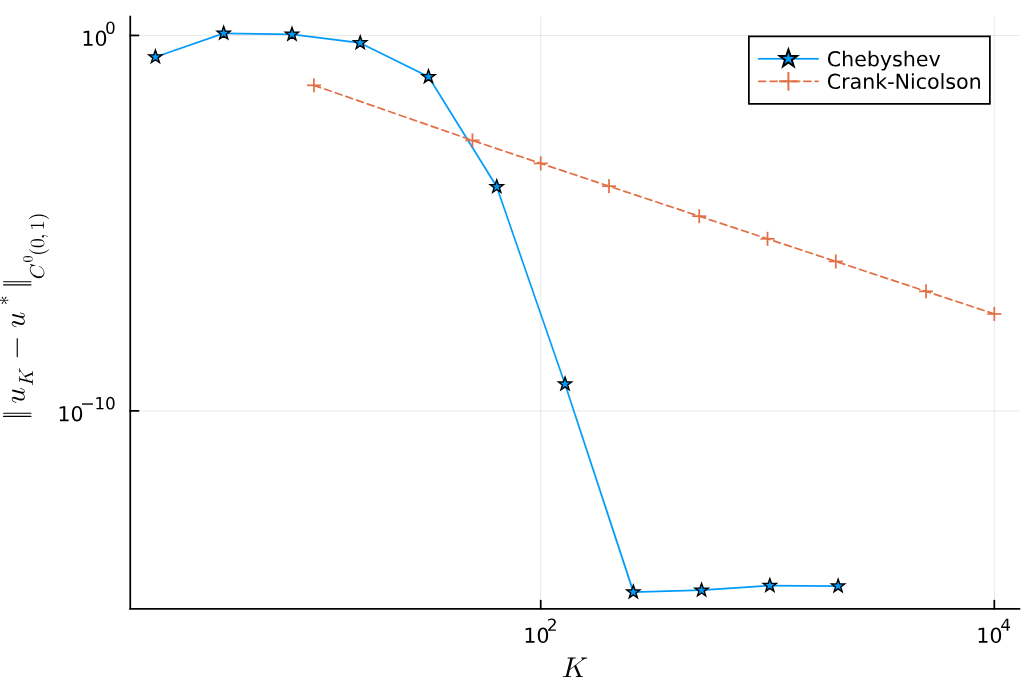}
		\caption{$\norm{u^*_{K, L} - u^*}_{\Czerosetsp{(-1, 1)}}$ for a Crank-Nicolson scheme and the global in time Chebyshev approach}
            \label{fig: EDO error plot}
	\end{subfigure}
	%Caption
	\caption{} \label{fig:Gaussian3D error}
\end{figure}
    
\subsection{Periodic Schrödinger equation} \label{subsec: Periodic Schrödinger equation}

We consider the Schrödinger equation on the two-dimensional torus $\TT^2 = (\RR / \ZZ)^2$, and on the time interval $(-\tau, \tau)$ :
\begin{equation} \label{equ: Schrodinger periodic}
    \begin{cases}
        i\partial_t u^* = (-\laplacian_{x,y} + V(t)) u^*,	\\
        u^*(0) = u_0,
    \end{cases}
\end{equation}
where for all $t\in \mathbb{R}$ and $(x,y)\in \TT^2$, $V(t, x, y) = \cos((2\pi(x - c_1 t)) + \cos(2\pi(y - c_2 t)) + \cos(2\pi (x - y))$. \\
A simple rescaling of the time variable shows that we can instead solve the following equation on $(-1, 1)$
\begin{equation}
\begin{cases}
    i\partial_t u^* = \tau (-\laplacian + V_\tau(t \tau)) u^*	\\
    u^*(0) = u_0
\end{cases}
\end{equation}
where $V_\tau(t, x, y) = V(t \tau, x, y)$.

The operator $-\laplacian$ has a well-known diagonal structure with respect to the orthonormal system of Fourier modes $\eexp_{k,l} = \eexp^{2i\pi kx}\eexp^{2i\pi ly}$, hence we will apply Theorem \ref{th: XX0 XX equality} and Corollary \ref{cor: Variational formulation v} with $H_0 = -\laplacian$.

Define the discrete set
\begin{equation}
    \Gamma_N
    =	\setspan \{\eexp_{k,l}\}_{\substack{-N/2 \leq k < N/2 \\ -N/2 \leq l < N/2}}
    \subset	\Ltwosp{\TT^2},
\end{equation}
and the orthogonal projector $\pi_N$ on $\Gamma_N$.
Since our goal here is to focus on the errors that arise from the time discretization, and not on the errors due to the truncation of the Fourier series, we will solve instead the ``discrete'' version of \eqref{equ: Schrodinger periodic} that is well-posed in $\Gamma_N$,
\begin{equation} \label{equ: Schrodinger periodic discrete}
    \begin{cases}
        i\partial_t u^*_N = \tau (-\laplacian + \pi_N V_\tau \pi_N) u^*_N,	\\
        u^*_N(0) = \pi_N u_0 \in \Gamma_N.
    \end{cases}
\end{equation}
Estimating the difference between the real solution $u^*$ to \eqref{equ: Schrodinger periodic} and the solution $u^*_N$ to \eqref{equ: Schrodinger periodic discrete} is another problem that can be studied independently, but we do not consider it here.

We know from Corollary \ref{cor: Variational formulation v} that we can, instead of \eqref{equ: Schrodinger periodic discrete}, solve the equivalent evolution equation
\begin{equation}
\begin{cases}
    i\partial_t v^*_N = \tau \eexp^{-it \tau \laplacian} (\pi_N V_\tau \pi_N) \eexp^{it \tau \laplacian} v^*_N,	\\
    v^*_N(0) = \pi_N u_0 \in \Gamma_N,
\end{cases}
\end{equation}
which is associated with the variational formulation
\begin{equation}
    \min_{v \in \Hone_w((-1,1), \Gamma_N)}
    \left( \abs{v(0) - \pi_N u_0}^2 + \int_{-1}^1 \diff t w(t) \abs{i\partial_t v - \tau \eexp^{-it \tau \laplacian} (\pi_N V_\tau \pi_N) \eexp^{it \tau \laplacian} v}^2 \right).
\end{equation}
We define the discrete subspace
\begin{equation}
    \Sigma_{K, N}
    :=  \set{\sum_{k=0}^{K-1} v_k U_k(t)}{(v_k)_{k=0}^{K-1} \subset \Gamma_N}
    \subset \Hone_w((-1,1), \Ltwosp{\TT^2}),
\end{equation}
and the discrete solution $v^*_{K,N}$ as the unique solution to the restricted minimization problem
\begin{equation}
    v^*_{K,N}
    =   \argmin_{v \in \Sigma_{K,N}} \left( \abs{v(0) - \pi_N u_0}^2 + \int_{-1}^1 \diff t w(t) \abs{i\partial_t v - \tau \eexp^{-it \tau \laplacian} (\pi_N V_\tau \pi_N) \eexp^{it \tau \laplacian} v}^2 \right).
\end{equation}

We employ a similar time discretization to that of Section~\ref{subsec: An ordinary differential equation}, and the matrix associated with the quadratic form is built similarly, except that it involves block entries instead of scalar entries.
The difference between the computed solution $v^*_{K,N}$ and $u^*_N$ is displayed in Figure \ref{fig: Ex2d error} for several values of $\tau$ and $N=64$. We compare it with an order 2 Crank-Nicolson scheme, and an order 4 Runge-Kutta scheme. As in the previous example, we observe a very fast convergence once $K$ reaches a certain threshold. Although the global space-time Chebyshev approach does not necessarily outperform classical time stepping scheme, it seems more efficient to reach very high precision when enough regularity is available, which is an expected behavior for a spectral method.
On Figure \ref{fig: Ex2d compute time}, we display the computation time for several values of $K$. We observe that the computational cost of the Chebyshev for a given value of $K$ has the same order of magnitude as the time-stepping schemes. The fact that it increases with $\tau$ is a consequence of the loss of quality of the preconditioner which leads to more iterations of the conjugate gradient algorithm used to solve the problem. Indeed, a smaller $\tau$ means that the equation is closer to the free dynamics that we use as a preconditioner.
Finally, Figure \ref{fig: Ex2d error wrt compute time} displays the errors of the different methods with respect to the computation cost. Since the computational cost of the different methods are very close, Figure \ref{fig: Ex2d error wrt compute time} looks very much like Figure \ref{fig: EDO error plot} with some corrections, and the same observations can be made.

\begin{figure}[h]
	\centering
	% First subfigure
	\begin{subfigure}[t]{0.49\textwidth}
		\centering
		\includegraphics[width=0.9\textwidth]{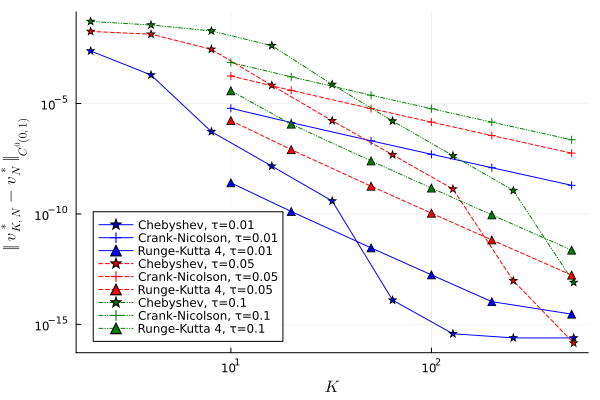}
        \caption{$\norm{v^*_{K, N} - v^*_N}_{\Czerosetsp{(-1, 1)}}$ for a Crank-Nicolson scheme, an order 4 Runge-Kutta scheme, and the global in time Chebyshev approach, for $N=64$}
            \label{fig: Ex2d error}
	\end{subfigure}
	\hfill
	% Second subfigure
	\begin{subfigure}[t]{0.49\textwidth}
		\centering
		\includegraphics[width=0.9\textwidth]{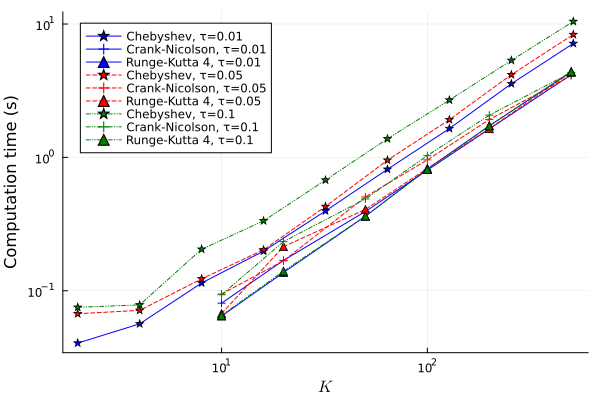}
        \caption{Computation time for a Crank-Nicolson scheme, an order 4 Runge-Kutta scheme, and the global in time Chebyshev approach, for $N=64$}
            \label{fig: Ex2d compute time}
	\end{subfigure}	
	% Third subfigure
	\begin{subfigure}[t]{0.5\textwidth}
		\includegraphics[width=0.9\textwidth]{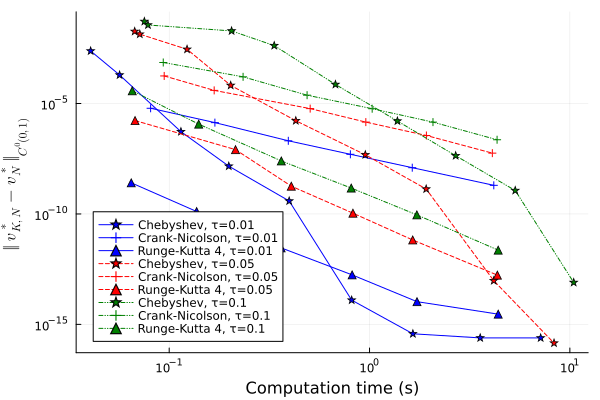}
        \caption{$\norm{v^*_{K, N} - v^*_N}_{\Czerosetsp{(-1, 1)}}$ with respect to computation time for a Crank-Nicolson scheme, an order 4 Runge-Kutta scheme, and the global in time Chebyshev approach, for $N=64$}
            \label{fig: Ex2d error wrt compute time}
	\end{subfigure}
	%Caption
	\caption{} \label{fig:Gaussian error}
\end{figure}

\section{Weak solutions to the Schrödinger equation} \label{sec: Weak solutions to the Schrödinger equation}

In this section we recall essential properties of weak solutions as defined in Definition \ref{def: definition weak solution}.
While the rest of the paper only mentioned the case of a time interval $I = (0, T)$, we consider here a general (possibly unbounded, unless stated otherwise) open time interval $J \subset \RR$ and we assume that $0 \in \setclosure J$.
We recall the definition of weak solutions in this context (see also Definition \ref{def: definition weak solution}):
\begin{definition} \label{def: Weak solutions unbounded}
    Let $u, f \in \Ltwosp{J, \HHilb}$. We say that $(i\partial_t - H - B(t)) u = f$ holds weakly if and only if
    \begin{equation}
        \forall \varphi \in \Cczerosetsp{J, \domain H} \cap \Cconesetsp{J, \HHilb}, \quad
        \dotprodparenthesis{u}{(i\partial_t - H) \varphi}_{\Ltwosp{J, \HHilb}}
        =	\dotprodparenthesis{f}{\varphi}_{\Ltwosp{J, \HHilb}},
    \end{equation}
    where $\domain H$ is equipped with the graph norm of $H$.
\end{definition}

We first deal with the case $B(t) = 0$.
\begin{proposition} \label{prop: Weak evolution derivative}
    Let $u \in \Ltwosp{J, \HHilb}$, and $v = \eexp^{itH} u$. The following are equivalent :
    \begin{enumerate}
        \item $(i\partial_t - H) u \in \Ltwosp{J, \HHilb}$.
        \item $i\partial_t v \in \Ltwosp{J, \HHilb}$.
    \end{enumerate}
    In particular, if 1. and 2. hold, $u$ and $v$ both belong to $\Czerosetsp{J, \HHilb}$.
    Moreover, in this case, $i\partial_t v = \eexp^{itH} (i\partial_t - H) u$, and
    \begin{equation}
        \forall t \in J, \quad
        u(t) = \eexp^{-itH} u(0) - i \int_0^t \diff s \eexp^{-i(t-s)H} (i\partial_t - H) u(s).
    \end{equation}
\end{proposition}
\begin{proof}
    Assume $(i\partial_t - H) u = f \in \Ltwosp{J, \HHilb}$, then, for any $\varphi \in \Ccinfsetsp{J, \domain{H^2}}$,
    \[
        \dotprodparenthesis{v}{i\partial_t \varphi}_{\Ltwosp{J, \HHilb}}
        =	\dotprodparenthesis{u}{\eexp^{-itH} i\partial_t \varphi}_{\Ltwosp{J, \HHilb}}
        =	\dotprodparenthesis{u}{(i\partial_t - H) \eexp^{-itH} i\partial_t \varphi}_{\Ltwosp{J, \HHilb}}
        =	\dotprodparenthesis{f}{\eexp^{-itH} \varphi}_{\Ltwosp{J, \HHilb}}.
    \]
    By density, this identity extends to all $\varphi \in \Cconesetsp{J, \HHilb}$, and we conclude that $v \in \Honesp{J, \HHilb}$ with $i\partial_t v = \eexp^{itH} f$.
    
    Conversely, assume $i\partial_t v = g \in \Ltwosp{J, \HHilb}$. Then for any $\varphi \in \Ccinfsetsp{J, \domain{H^2}}$,
    \[
        \dotprodparenthesis{u}{(i\partial_t - H) \varphi}_{\Ltwosp{J, \HHilb}}
        =	\dotprodparenthesis{v}{\eexp^{itH} (i\partial_t - H) \varphi}_{\Ltwosp{J, \HHilb}}
        =	\dotprodparenthesis{v}{i\partial_t \eexp^{itH} \varphi}_{\Ltwosp{J, \HHilb}}
        =	\dotprodparenthesis{g}{\eexp^{itH} \varphi}_{\Ltwosp{J, \HHilb}}.
    \]
    By density, the identity extends to all $\varphi \in \Cczerosetsp{J, \domain H} \cap \Cconesetsp{J, \HHilb}$, so $(i\partial_t - H) u = \eexp^{-itH} g \in \Ltwosp{J, \HHilb}$.
    Moreover, we then have for all $t \in J$
    \[
    \begin{aligned}
        u(t)
        &=	\eexp^{-itH} v(t)
        =	\eexp^{-itH} (v(0) - i \int_0^t \diff s g(s))	\\
        &=	\eexp^{-itH} u(0) - i \int_0^t \diff s \eexp^{-i(t-s)H} (i\partial_t - H) u(s).
    \end{aligned}
    \]
\end{proof}

\begin{corollary} \label{cor: Weak solutions unbounded}
    For any $f \in \Ltwosp{J, \HHilb}$ and $u_0 \in \HHilb$, there exists exactly one $u^* \in \Ltwosp{J, \HHilb}$ such that $(i\partial_t - H) u^* = f$ and $u^*(0) = u_0$.
\end{corollary}
\begin{proof}
    \underline{Existence:} Let $f \in \Ltwosp{J, \HHilb}$. Define
    \[
        v^*(t) = u_0 - i \int_0^t \diff s \eexp^{isH} f(s),
    \]
    and $u^* = \eexp^{-itH} v^*$. Then, by Proposition \ref{prop: Weak evolution derivative}, $(i\partial_t - H) u^* = \eexp^{-itH} \eexp^{itH} i\partial_t v^* = f$, and $u^*(0) = u_0$.
    
    \underline{Uniqueness:} Assume $(i\partial_t - H) u^*_1 = (i\partial_t - H) u^*_2 = f \in \Ltwosp{J, \HHilb}$ and $u^*_1(0) = u^*_2(0) = u_0$. Define $v^*_j = \eexp^{itH} u^*_j$ ($j=1, 2$), we then have $\partial_t v^*_1 = \partial_t v^*_2$. Since $v^*_1(0) = v^*_2(0) = u_0$, uniqueness follows.
\end{proof}

We now add a time-dependent part $B(t)$, assumed to be a strongly continuous family of uniformly bounded self-adjoint operators (that is, for any $t \in J$, $B(t)$ is a bounded self-adjoint operator, and $\sup_{t \in J} \norm{B(t)} < \infty$).
\begin{proposition} \label{prop: Weak solutions time dependent}
    Assume $J$ is \textbf{bounded}.
    For any $f \in \Ltwosp{J, \HHilb}$ and $u_0 \in \HHilb$ there exist a unique $u^* \in \Ltwosp{J, \HHilb}$ such that $(i\partial_t - H - B(t)) u^* = f$ and $u(0) = u_0$.
    Furthermore, the following continuity estimate holds:
    \begin{equation} \label{equ: XX continity estimate time dependent}
        \norm{u^*}_{\Czerosetsp{J, \HHilb}}
        \leq    \sqrt 2 \left( \abs{u_0}^2 + T \norm f_{\Ltwosp{J, \HHilb}}^2 \right)^{\frac 1 2}
    \end{equation}
\end{proposition}
\begin{proof}
    The result is proved by a standard fixed-point argument for ODEs.
    For simplicity, we assume here $J = (0, T)$ for some $T > 0$. The extension to negative times is straightforward.
    Set
    \[
        M = \sup_{t \in J} \norm{B(s)}.
    \]
    We also define, for $\mu > 0$,
    \[
        \forall u \in \Ltwosp{J, \HHilb}, \quad
        \norm u_{\Ltwosp{J, \HHilb}}' 
        =   \left( \int_0^T \diff t \eexp^{-\mu t} \abs{u(t)}^2 \right)^{\frac 1 2},
    \]
    which is a norm on $\Ltwosp{J, \HHilb}$ equivalent to $\norm \cdot_{\Ltwosp{J, \HHilb}}$.
    
    For any $u \in \Ltwosp{J, \HHilb}$ and $t \in J$, set
    \[
        \Phi(u)(t)
        =   \eexp^{-itH} u_0 - i \int_0^t \diff s \eexp^{-i(t-s)H} (f(s) - B(s) u(s)).
    \]
    Clearly $\Phi(u) \in \Ltwosp{J, \HHilb}$, and for any $u, v \in \Ltwosp{J, \HHilb}$,
    \[
    \begin{aligned}
        {\norm{\Phi(u) - \Phi(v)}_{\Ltwosp{J, \HHilb}}'}^2
        &=   \int_0^T \diff t \eexp^{-\mu t} \abs{\int_0^t \diff s \eexp^{-i(t-s)H} B(s)(v(s) - u(s))}^2    \\
        &\leq   M^2 T \int_0^T \diff t \eexp^{-\mu t} \int_0^T \diff s \abs{u(s) - v(s)}^2   \\
        &\leq   M^2 T \int_0^T \diff s \abs{u(s) - v(s)}^2 \int_s^\infty \diff t \eexp^{-\mu t} \\
        &\leq   \frac{M^2 T}{\mu} \int_0^T \diff s \eexp^{-\mu s} \abs{u(s) - v(s)}^2.
    \end{aligned}
    \]
    Choosing $\mu$ such that $\frac{M^2 T}{\mu} < 1$, this proves that $\Phi$ is contractive for $\norm \cdot_{\Ltwosp{J, \HHilb}}'$, hence admits a unique fixed point $u^*$ which satisfies for all $t \in \HHilb$
    \begin{equation} \label{equtmp: u-star fixed point}
        u^*(t)
        =   \eexp^{-itH} u_0 - i \int_0^t \diff s \eexp^{-i(t-s)H} (f(s) - B(s) u^*(s)),
    \end{equation}
    which, by Proposition \ref{prop: Weak evolution derivative}, is equivalent to $(i\partial_t - H - B(t)) u = f$ and $u(0) = u_0$.

    To prove the continuity estimate, we first assume $f \in \Czerosetsp{J, \HHilb}$, and write for any $t \in J$
    \[
        \abs{u^*(t)}^2
        =   \abs{u_0 - i \int_0^t \diff s \eexp^{isH} (f(s) - B(s) u^*(s))}^2,
    \]
    from which we deduce
    \[
    \begin{aligned}
        \frac{\diff{}}{\diff t} \abs{u^*(t)}^2
        &=  2 \Re \dotprodbracket{u_0 - i \int_0^t \diff s \eexp^{isH} (f(s) - B(s) u^*(s))}{-i \eexp^{itH}(f(t) - B(t) u^*(t))}   \\
        &=  2 \Im \dotprodbracket{u^*(t)}{f(t) - B(t) u^*(t)}
        =   2 \Im \dotprodbracket{u^*(t)}{f(t)}.
    \end{aligned}
    \]
    This yields
    \[
        \forall t \in J, \quad
        \abs{\frac{\diff{}}{\diff t} \abs{u^*(t)}^2}
        \leq    2 \abs{u^*(t)} \abs{f(t)}.
    \]
    Let $\varepsilon > 0$ and $w_\varepsilon(t) = \varepsilon + \abs{u_0} + \int_0^t \diff s \abs{f(s)}$.
    Let $K_\varepsilon = \quickset{w_\varepsilon \geq \abs{u^*}} \subset \setclosure J$, which is closed (by continuity of $w$ and $u^*$) and contains $0$.
    Now assume $t_0 := \inf \setclosure J \setminus K_\varepsilon < T$. By continuity, it holds $\abs{u^*(t_0)} = w(t_0) \geq \varepsilon > 0$, and there exists a $\delta > 0$ such that $[t_0, t_0 + \delta] \subset J$ and $\abs{u^*(t)}$ for every $t \in [t_0, t_0 + \delta]$. Therefore,
    \[
        \forall t \in [t_0 - \delta, t_0 + \delta], \quad
        2 \abs{u^*(t)} \frac{\diff{}}{\diff t} \abs{u^*(t)}
        =   \frac{\diff{}}{\diff t} \abs{u^*(t)}^2
        \leq    2 \abs{f(t)} \abs{u^*(t)},
    \]
    which yields
    \[
        \forall t \in [t_0, t_0 + \delta], \quad
        \frac{\diff{}}{\diff t} \abs{u^*(t)}
        \leq    \abs{f(t)}.
    \]
    It follows by integration that
    \[
        \forall t \in [t_0, t_0 + \delta], \quad
        \abs{u^*(t)}
        \leq    \abs{u^*(t_0)} + \int_{t_0}^t \diff s \abs{f(s)}
        =   w(t),
    \]
    which contradicts the definition of $t_0$.
    Therefore,
    \[
        \forall t \in \setclosure J, \quad
        \abs{u^*(t)}
        \leq    w_\varepsilon(t)
        \leq    \varepsilon + \abs{u_0} + \sqrt T \norm f_{\Ltwosp{J, \HHilb}}
        \leq    \varepsilon + \sqrt 2 \left( \abs{u_0}^2 + T \norm f_{\Ltwosp{J, \HHilb}}^2 \right)^{\frac 1 2}.
    \]
    Since this holds for any $\varepsilon > 0$, we obtain the desired inequality.

    For general $f \in \Ltwosp{J, \HHilb}$, we take a sequence $(f_n)_{n \geq 0}$ such that $\norm{f_n - f}_{\Ltwosp{J, \HHilb}} \to 0$, and define $u^*_n$ the unique element of $\Ltwosp{J, \HHilb}$ such that
    \begin{equation} \label{equtmp: u-star n fixed point}
        u^*_n(t)
        =   \eexp^{-itH} u_0 - i \int_0^t \diff s \eexp^{-i(t-s)H} (f_n(s) - B(s) u^*_n(s)).
    \end{equation}
    As earlier, we compare $u^*_n$ and $u^*$ in term of $\norm \cdot'_{\Ltwosp{J, \HHilb}}$, and obtain with \eqref{equtmp: u-star fixed point} and \eqref{equtmp: u-star n fixed point}
    \[
        {\norm{u_n^* - u^*}'_{\Ltwosp{J, \HHilb}}}^2
        \leq    C \norm{f_n - f}_{\Ltwosp{J, \HHilb}}^2 + \frac{M^2 T}{\mu} {\norm{u_n^* - u^*}'_{\Ltwosp{J, \HHilb}}}^2,
    \]
    yielding
    \[
        \norm{u_n^* - u^*}'_{\Ltwosp{J, \HHilb}}
        \leq    C \norm{f_n - f}_{\Ltwosp{J, \HHilb}},
    \]
    for some constant $C > 0$.
    This proves that $u^*_n \to u^*$ in $\Ltwosp{J, \HHilb}$, and using \eqref{equtmp: u-star fixed point} and \eqref{equtmp: u-star n fixed point} again we deduce that $u^*_n \to u^*$ in $\Czerosetsp{J, \HHilb}$. The result appears by taking $n \to \infty$ in
    \[
        \forall n \geq 0, \quad
        \norm{u^*_n}_{\Czerosetsp{J, \HHilb}}
        \leq    \sqrt 2 \left( \abs{u_0}^2 + T \norm{f_n}_{\Ltwosp{J, \HHilb}}^2 \right)^{\frac 1 2}.
    \]
\end{proof}

%\printbibliography
\bibliographystyle{plain}
\bibliography{biblio}
	
\end{document}